\documentclass[11pt,a4paper]{article}

\usepackage[utf8]{inputenc}
\PassOptionsToPackage{hyphens}{url}
\usepackage{amsmath,amssymb,amsthm}
\usepackage{mathtools}
\usepackage{graphicx}
\usepackage{booktabs}
\usepackage{array}
\usepackage[hypertexnames=false,hidelinks,pdfauthor={Kyle Elliott Mathewson},pdftitle={Moments of Magic Squares: Line-Sum Conditions, a Permutation Threshold, and Random Energies}]{hyperref}
\usepackage{geometry}
\usepackage{float}

\makeatletter
\def\Cref@getprefix#1:#2\relax{#1}
\def\Cref@prefix@sec{sec}
\def\Cref@prefix@fig{fig}
\def\Cref@prefix@tab{tab}
\def\Cref@prefix@eq{eq}
\def\Cref@prefix@thm{thm}
\def\Cref@prefix@prop{prop}
\def\Cref@prefix@lem{lem}
\def\Cref@prefix@cor{cor}
\def\Cref@prefix@conj{conj}
\def\Cref@prefix@def{def}
\def\Cref@prefix@rem{rem}
\def\Cref@prefix@ex{ex}

\def\Cref@type@cap#1{%
  \edef\Cref@tmp{#1}%
  \ifx\Cref@tmp\Cref@prefix@sec Section%
  \else\ifx\Cref@tmp\Cref@prefix@fig Figure%
  \else\ifx\Cref@tmp\Cref@prefix@tab Table%
  \else\ifx\Cref@tmp\Cref@prefix@eq Equation%
  \else\ifx\Cref@tmp\Cref@prefix@thm Theorem%
  \else\ifx\Cref@tmp\Cref@prefix@prop Proposition%
  \else\ifx\Cref@tmp\Cref@prefix@lem Lemma%
  \else\ifx\Cref@tmp\Cref@prefix@cor Corollary%
  \else\ifx\Cref@tmp\Cref@prefix@conj Conjecture%
  \else\ifx\Cref@tmp\Cref@prefix@def Definition%
  \else\ifx\Cref@tmp\Cref@prefix@rem Remark%
  \else\ifx\Cref@tmp\Cref@prefix@ex Example%
  \else Ref.%
  \fi\fi\fi\fi\fi\fi\fi\fi\fi\fi\fi\fi
}

\def\Cref@type@low#1{%
  \edef\Cref@tmp{#1}%
  \ifx\Cref@tmp\Cref@prefix@sec section%
  \else\ifx\Cref@tmp\Cref@prefix@fig figure%
  \else\ifx\Cref@tmp\Cref@prefix@tab table%
  \else\ifx\Cref@tmp\Cref@prefix@eq equation%
  \else\ifx\Cref@tmp\Cref@prefix@thm theorem%
  \else\ifx\Cref@tmp\Cref@prefix@prop proposition%
  \else\ifx\Cref@tmp\Cref@prefix@lem lemma%
  \else\ifx\Cref@tmp\Cref@prefix@cor corollary%
  \else\ifx\Cref@tmp\Cref@prefix@conj conjecture%
  \else\ifx\Cref@tmp\Cref@prefix@def definition%
  \else\ifx\Cref@tmp\Cref@prefix@rem remark%
  \else\ifx\Cref@tmp\Cref@prefix@ex example%
  \else ref.%
  \fi\fi\fi\fi\fi\fi\fi\fi\fi\fi\fi\fi
}

\DeclareRobustCommand{\Cref}[1]{%
  \begingroup
  \edef\Cref@label{#1}%
  \edef\Cref@prefix{\expandafter\Cref@getprefix\Cref@label:\relax}%
  \hyperref[#1]{\Cref@type@cap{\Cref@prefix}~\ref*{#1}}%
  \endgroup
}
\DeclareRobustCommand{\cref}[1]{%
  \begingroup
  \edef\Cref@label{#1}%
  \edef\Cref@prefix{\expandafter\Cref@getprefix\Cref@label:\relax}%
  \hyperref[#1]{\Cref@type@low{\Cref@prefix}~\ref*{#1}}%
  \endgroup
}
\makeatother

\makeatletter
\long\def\@makecaption#1#2{\vskip\abovecaptionskip\small\sbox\@tempboxa{#1: #2}\ifdim\wd\@tempboxa>\hsize #1: #2\par\else\global\@minipagefalse\hb@xt@\hsize{\hfil\box\@tempboxa\hfil}\fi\vskip\belowcaptionskip}
\makeatother

\theoremstyle{plain}
\newtheorem{theorem}{Theorem}[section]
\newtheorem{lemma}[theorem]{Lemma}
\newtheorem{proposition}[theorem]{Proposition}
\newtheorem{corollary}[theorem]{Corollary}
\newtheorem{conjecture}[theorem]{Conjecture}

\theoremstyle{definition}
\newtheorem{definition}[theorem]{Definition}

\theoremstyle{remark}
\newtheorem{remark}[theorem]{Remark}

\newcommand{\R}{\mathbb{R}}
\newcommand{\Z}{\mathbb{Z}}

\newcommand{\Cov}{\mathrm{Cov}}
\newcommand{\Var}{\mathrm{Var}}
\newcommand{\Efull}{E_{\mathrm{full}}}
\newcommand{\Efour}{E_4}
\newcommand{\E}{\mathbb{E}}
\newcommand{\tr}{\operatorname{tr}}

\title{\textbf{Moments of Magic Squares:\\ Line-Sum Conditions, a Permutation Threshold,\\ and Random Energies}}

\author{Kyle Elliott Mathewson\\
\textit{Department of Psychology, Faculty of Science, University of Alberta,}\\
\textit{Edmonton, Alberta T6G 2E9, Canada}\\
\textit{Corresponding author email:} \texttt{kyle.mathewson@ualberta.ca}}

\date{\today}

\begin{document}

\maketitle

\begin{abstract}
Every semi-magic square of order $n$ satisfies the axial moment identities $\sum_{i,j}s_{ij}j^m=\sum_{i,j}s_{ij}i^m=M(n)\sum_jj^m$ for all $m\ge1$, where $M(n)=n(n^2+1)/2$. For real matrices the converse requires $1\le m\le n-1$, whereas for arrangements of $\{1,\dots,n^2\}$ fewer moments suffice once $n\ge14$. We define the \emph{permutation threshold} $k^*(n)$ as the least $k$ such that every arrangement satisfying these identities for $1\le m\le k$ is semi-magic. In terms of the column-sum deviations $\gamma_j$, the degree-$k$ conditions state that $\sum_j\gamma_jt^j$ is divisible by $(t-1)^{k+1}$, linking $k^*(n)$ to the Prouhet--Tarry--Escott problem. While Schur's bound on real zeros gives $k^*(n)\le\sqrt{2(n-1)\log(n^4/4)}$, we prove the converse realisability bound $k^*(n)\ge(\sqrt2-o(1))\sqrt n$ (and $k^*(n)\ge(\sqrt6-o(1))\sqrt n$ when $4\mid n$) by perturbing semi-magic squares with complementary row pairs via horizontal swaps. We also show that $k^*(n)=n-1$ for $3\le n\le13$, $k^*(n)=n-2$ for $14\le n\le17$, and $k^*(n)\le n-3$ for $n\ge18$, and determine $k^*(n)$ for $18\le n\le40$ with computer assistance. Finally, at the $k=1$ level of the hierarchy, we specialise the Mantel--Hubert quadratic assignment variance formula to obtain exact means, variances, and limit laws under a uniform random arrangement for the full line-sum energy $\Efull$, for the four-term energy $\Efour$ combining the first moments and the two diagonals, and for Loly's moment of inertia, and find that $|L_4|=14{,}639{,}440$ arrangements of order $4$ satisfy $\Efour=0$.


\end{abstract}

\section{Introduction}
\label{sec:introduction}

Magic squares are among the oldest objects of recreational mathematics: the $3\times3$ Lo Shu appears in Chinese sources more than two millennia old, a $4\times4$ square is engraved in D\"urer's \emph{Melencolia~I} (1514), and constructions for every order have been known since the seventeenth century \cite{andrews1960magic,pickover2002zen}.
The modern literature treats them through enumeration and the geometry of the associated polytopes \cite{stanley1973magic,ehrhart1973magic,ahmed2003polyhedral,beck2006insideout,beck2010enumeration}, through spectral and matrix properties \cite{loly2009spectra,nordgren2012properties,nordgren2014spectra,cameron2013signatura}, and, in work of Loly and collaborators, through physical invariants such as the moment of inertia and the multipole moments \cite{loly2004invariance,rogers2004inertia,rogers2005multipole}.

\begin{definition}[Magic square]
\label{def:magic_square}
An \emph{arrangement} of order $n$ is a bijection $S=(s_{ij})$ from the cells of the $n\times n$ grid onto $\{1,\dots,n^2\}$.
It is a \emph{magic square} if every row, every column and both main diagonals sum to the magic constant $M(n)=n(n^2+1)/2$, and \emph{semi-magic} if the rows and columns do.
\end{definition}

We centre both the grid and the entries (\Cref{def:coords}): the cell $(i,j)$ sits at $(x_j,y_i)=(j-\frac{n-1}2,\frac{n-1}2-i)$ and carries the centred value $z_{ij}=s_{ij}-\frac{n^2+1}2$.
A \emph{test function} is a real function $\varphi$ on the cells, paired with $z$ through the normalised inner product $\langle\varphi,z\rangle=n^{-2}\sum_{i,j}\varphi_{ij}z_{ij}$.
Requiring a line $L$ to sum to $M(n)$ is equivalent to requiring $z$ to be orthogonal to the centred indicator of $L$, so an arrangement is magic if and only if $z$ is orthogonal to a fixed $2n$-dimensional subspace $W$ of test functions—recovering the classical fact that real magic squares of order $n\ge3$ form a vector space of dimension $n^2-2n$ \cite{ward1980vector}—while the squared length of the line-deviation vector, $\sum_L\delta_L^2$, is the energy functional studied in the Monte Carlo work of Pinn and Wieczerkowski \cite{pinn1998number}.

In the companion paper \cite{mathewson_gems}, the triples $(x_j,y_i,z_{ij})$ are viewed geometrically as a lifted point set in $\R^3$ whose convex hull is the \emph{Magic Gem} of the square. Here we ask a complementary algebraic and probabilistic question: \emph{how closely do low-degree polynomial moments of $z$ approximate the full line-sum conditions when $S$ is constrained to be a permutation of $\{1,\dots,n^2\}$?} This question leads to two interconnected problems—one extremal, one quantitative—that form the two halves of the paper.

\paragraph{The permutation threshold.}
Replacing the line indicators by the powers $x^m$ and $y^m$, $1\le m\le k$, gives the \emph{degree-$k$ conditions} $\sum_{i,j}z_{ij}x_j^m=\sum_{i,j}z_{ij}y_i^m=0$; equivalently, the moments of orders $1,\dots,k$ of the vector of column-sum deviations and of the vector of row-sum deviations vanish (\Cref{lem:poly_form}).
They hold for every semi-magic square, and for real-valued $z$ they imply equal row and column sums if and only if $k\ge n-1$ (Vandermonde).
For arrangements, integrality and the bounded size of the line sums lower this threshold.
Let $k^*(n)$ be the least $k$ such that every arrangement satisfying the degree-$k$ conditions is semi-magic.
A \emph{witness} for $k^*(n)\ge K$ is an arrangement that satisfies the degree-$(K-1)$ conditions and is not semi-magic.
The degree-$k$ conditions say that the polynomial $\sum_j\gamma_jt^j$ formed from the column deviations is divisible by $(t-1)^{k+1}$, and similarly for the rows, so that the positive and negative parts of $\gamma$ form a solution of a Prouhet--Tarry--Escott system \cite{wright1959prouhet,borwein2002computational}.
Our main result is the lower bound in \Cref{thm:perm_threshold}(b); together with Schur's bound, \Cref{thm:perm_threshold}(a), it gives
\[
(\sqrt2-o(1))\sqrt n\ \le\ k^*(n)\ \le\ \sqrt{2(n-1)\log(n^4/4)}\ \le\ 2\sqrt2\,\sqrt{n\log n},
\]
with $\sqrt6$ in place of $\sqrt2$ when $4\mid n$.
We also prove $k^*(n)=n-1$ for $3\le n\le13$, $k^*(n)=n-2$ for $14\le n\le17$ and $k^*(n)\le n-3$ for all $n\ge18$, with short proofs in which the lower bounds rest on explicit witness arrangements (printed in Appendix~\ref{sec:app_witnesses} for $n\le13$), and, separately and with computer assistance, we determine $k^*(n)$ for $18\le n\le40$ (\Cref{tab:kstar}).

The upper bound does not use that $\gamma$ comes from an arrangement, only that it is a nonzero integer vector with $\|\gamma\|_1\le n^4/4$.
For such vectors it is Schur's bound on the number of real zeros \cite{schur1933,erdosturan1950}; the same order follows from Mignotte's lower bound for the height of an integer polynomial with a zero of high order at $1$ \cite{mignotte1982,amoroso1995} and from \cite[Thm.~2.1]{BEK1999}.
It is below $n-1$ from $n=24$ on and improves on $k^*(n)\le n-3$ from $n=30$ on.
In the other direction, Amoroso's upper bound \cite[Thm.~2]{amoroso1995}, obtained from the Bombieri--Vaaler form of Siegel's lemma \cite{bombieri1987low}, gives, for large $n$, such vectors with a zero of order $(\sqrt{24}-o(1))\sqrt n$ at $1$, and between the orders $\sqrt{24n}$ and $\sqrt{8n\log n}$ lies the known gap, by a factor of order $\sqrt{\log n}$, between the constructions and the bounds for integer polynomials with a zero of prescribed order at $1$ \cite{mignotte1982,bombieri1987low,amoroso1995,BEK1999}.
The new ingredient is the realisability direction behind the lower bound: integer vectors of this kind must be column-deviation vectors of arrangements whose row sums all equal $M(n)$.
We obtain them from semi-magic squares with many pairs of complementary rows, in which swaps realise prescribed column deviations, combined with the pigeonhole principle.
Whether $k^*(n)$ grows like $\sqrt n$ or like $\sqrt{n\log n}$ remains open, and the computed values for $14\le n\le100$ (exact for $n\le40$, upper bounds beyond) do not decide it (\Cref{fig:kstar}).
Let $U(n)$ be the least $k$ for which the lattice of admissible $\gamma$ contains no nonzero vector obeying the size constraints satisfied by column deviations, so that $k^*(n)\le U(n)$.  Then $k^*(n)=U(n)$ for $n\le40$, witnesses prove $k^*(n)\ge U(n)$ for $41\le n\le46$ (with $U(n)$ computed but not certified), and we conjecture that $k^*(n)=U(n)$ for all $n$ (\Cref{conj:lattice}).

\paragraph{Random arrangements and the $k=1$ relaxation.}
While \Cref{sec:theory} asks when degree-$k$ axial moments force semi-magicness, Sections~\ref{sec:random_energy} and~\ref{sec:computations} examine the lowest non-trivial level $k=1$ of the moment hierarchy across the entire permutation group $S_{n^2}$.
For a uniformly random arrangement, each inner product $\langle\varphi,z\rangle$ is a linear permutation statistic, so both the full line-sum energy $\Efull=N^{-2}\sum_L\delta_L^2$ ($N=n^2$, all $2n+2$ lines) and the four-term energy $\Efour$—which replaces the $2n$ row and column sums by the two degree-$1$ axial moments $\langle x,z\rangle$ and $\langle y,z\rangle$ while keeping the two main diagonals (\Cref{def:energies})—are quadratic forms in a random permutation of the centred values, i.e.\ quadratic assignment statistics.
Specialising the permutation variance formula of Mantel \cite{mantel1967} and Hubert \cite{hubert1987assignment} to symmetric matrices with zero row sums via M\"obius inversion on the partition lattice (\Cref{prop:var_identity}), we express the variance in terms of three matrix trace invariants and obtain exact closed-form means and variances for both energies at every order $n$ (\Cref{thm:var_closed}).
Hoeffding's combinatorial central limit theorem then shows that $\Efour/\E[\Efour]\xrightarrow{d}\mathrm{Exp}(1)$ as $n\to\infty$ (\Cref{cor:random_energy_clt}), dominated by the two first-moment terms $\langle x,z\rangle^2+\langle y,z\rangle^2$.
As a byproduct, the same inner-product framework yields the exact mean and variance of Loly's moment-of-inertia invariant $I(S)=\sum s_{ij}(x_j^2+y_i^2)$ over random permutations (\Cref{cor:loly_null}), showing that Loly's constant $n^2(n^4-1)/12$ for semi-magic squares \cite{loly2004invariance} is simply the null expectation $\E[I(S)]$.

Finally, in \Cref{sec:computations} we study the zero set $L_n = \{\Efour = 0\}$ for small $n$. Although $k^*(3)=2$ for axial moments alone, we give a short elementary proof that adding the two main diagonals at $n=3$ collapses the $k=1$ zero set $L_3$ to the $8$ magic squares (\Cref{prop:n3_lowmode}). At order $n=4$, by contrast, a meet-in-the-middle enumeration shows that exactly $|L_4|=14{,}639{,}440$ of the $16!$ arrangements satisfy $\Efour=0$, of which $7{,}040$ are magic (\Cref{prop:L4}), in close agreement with a Gaussian local-limit heuristic.

\paragraph{Organisation.}
\Cref{sec:background} sets up the centred coordinates, the orthogonal decomposition onto the line space $W$, and the classical permutation-statistic tools, summarised in \Cref{tab:notation}.
\Cref{sec:theory} treats the polynomial moment hierarchy and proves \Cref{thm:perm_threshold} on the permutation threshold $k^*(n)$ (with the complementary-row square constructions in Appendix~\ref{sec:app_twin} and explicit witnesses in Appendix~\ref{sec:app_witnesses}).
\Cref{sec:random_energy} establishes the exact mean, variance and limit law of $\Efull$, $\Efour$ and $I(S)$ for random arrangements (with the explicit trace evaluations in Appendix~\ref{sec:app_variance}), \Cref{sec:computations} analyses the zero sets $L_3$ and $L_4$ and single-transposition energy gaps, and \Cref{sec:discussion} collects open questions. All computational verifications and exact certificates are documented in Appendix~\ref{sec:appendix}.

\section{Notation and tools}
\label{sec:background}

\subsection{Centred coordinates}
\label{sec:bg_coords}

Throughout, $n\ge3$ and $N=n^2$.  Cells are pairs $c=(i,j)$ with $0\le i,j\le n-1$ ($i$ the row from the top, $j$ the column from the left).

\begin{definition}[Centred coordinates]
\label{def:coords}
To the cell $(i,j)$ of an arrangement $S$ assign
\begin{equation}
x_{ij}=x_j=j-\frac{n-1}{2},\qquad y_{ij}=y_i=\frac{n-1}{2}-i,\qquad z_{ij}=s_{ij}-\frac{N+1}{2}.
\label{eq:coords}
\end{equation}
We call $z=z(S)$ the \emph{centred entries} of $S$, and regard $x$, $y$ and $z$ as functions on the cells.
\end{definition}

The function $z$ is the vertical coordinate of the lifted points $(x_c,y_c,z_c)$ studied in the companion paper \cite{mathewson_gems}.
A \emph{line} is one of the $n$ rows $R_i$, the $n$ columns $C_j$, the main diagonal $D_{\mathrm m}=\{i=j\}$ or the anti-diagonal $D_{\mathrm a}=\{i+j=n-1\}$, and $\mathbf 1_L$ denotes the indicator of a line $L$.
Each line has $n$ cells, and since $M(n)=n\cdot\frac{N+1}2$ the \emph{deviation} of a line is
\begin{equation}
\delta_L(S)=\sum_{c\in L}s_c-M(n)=\sum_{c\in L}z_c .
\label{eq:line_dev}
\end{equation}
We write $\rho_i$, $\gamma_j$, $\delta_{\mathrm m}$, $\delta_{\mathrm a}$ for the deviations of row $i$, column $j$ and the two diagonals; $\sum_i\rho_i=\sum_j\gamma_j=0$.
On the space $\mathcal F=\R^{\text{cells}}$ we use $\langle f,g\rangle=\frac1N\sum_cf_cg_c$ and $\|f\|=\langle f,f\rangle^{1/2}$; $\bar f=\langle f,\mathbf 1\rangle$ is the mean of $f$ and $\tilde f=f-\bar f\mathbf 1$ its centred version.
For vectors in $\R^n$, such as $\rho$ and $\gamma$, $\|\cdot\|_1$ and $\|\cdot\|_2$ denote the usual $\ell^1$ and Euclidean norms.

\begin{lemma}
\label{lem:variances}
For every arrangement, $\bar x=\bar y=\bar z=0$, $\|x\|^2=\|y\|^2=\frac{n^2-1}{12}$, $\|z\|^2=\frac{N^2-1}{12}$, $\langle x,y\rangle=0$, and $\langle\mathbf 1_L,z\rangle=\delta_L/N$ for every line $L$.
\end{lemma}

\begin{proof}
$x$, $y$ and $z$ take $n$, $n$ and $N$ equally spaced values centred at $0$, each equally often; $\langle x,y\rangle=\frac1N\sum_jx_j\sum_iy_i=0$; the last claim is \eqref{eq:line_dev}.
\end{proof}

\subsection{The line space and the energies}
\label{sec:theory_projection}

Let $\tilde{\mathbf 1}_L=\mathbf 1_L-\frac1n\mathbf 1$ be the centred indicator of a line and
$W_R=\operatorname{span}\{\tilde{\mathbf 1}_{R_i}\}$, $W_C=\operatorname{span}\{\tilde{\mathbf 1}_{C_j}\}$, $W_D=\operatorname{span}\{\tilde{\mathbf 1}_{D_{\mathrm m}},\tilde{\mathbf 1}_{D_{\mathrm a}}\}$, $W=W_R+W_C+W_D\subset\mathbf 1^\perp$.
$P_W$ denotes the orthogonal projection onto $W$.

\begin{lemma}[The line space]
\label{lem:W_structure}
Let $n\ge3$.  For lines $L,L'$, $N\langle\tilde{\mathbf 1}_L,\tilde{\mathbf 1}_{L'}\rangle=|L\cap L'|-1$.  Consequently:
\begin{enumerate}
\item $W_R$ consists of the zero-mean functions of the row index and $W_C$ of the zero-mean functions of the column index; both have dimension $n-1$, and the orthogonal projection of $z$ onto them is $(i,j)\mapsto\rho_i/n$, resp.\ $\gamma_j/n$.
\item $W_R$, $W_C$, $W_D$ are mutually orthogonal, and the Gram matrix of $(\tilde{\mathbf 1}_{D_{\mathrm m}},\tilde{\mathbf 1}_{D_{\mathrm a}})$ is $\Gamma_D/N$ with $\Gamma_D=\bigl(\begin{smallmatrix}n-1&\epsilon_n\\ \epsilon_n&n-1\end{smallmatrix}\bigr)$, $\epsilon_n=0$ for odd and $-1$ for even $n$; its eigenvalues are $n-1,n-1$ (odd $n$) and $n-2,n$ (even $n$).
\item $W=W_R\oplus W_C\oplus W_D$ has dimension $2n$, and $z\perp W$ if and only if $S$ is magic.
\end{enumerate}
\end{lemma}

\begin{proof}
Each line has $n$ of the $N$ cells, which gives the inner products: a row meets every column and each diagonal in one cell, two distinct rows (or columns) are disjoint, and the diagonals meet in the centre cell if and only if $n$ is odd.  The projection onto the functions of the row index replaces $z$ by its row means $\rho_i/n$.  Finally $\langle\tilde{\mathbf 1}_L,z\rangle=\delta_L/N$ by \Cref{lem:variances}.
\end{proof}

Part~3 is equivalent to the classical fact that real magic squares of order $n\ge3$ with a prescribed magic sum (such as $0$ in $\mathbf 1^\perp$) form a subspace of dimension $N-2n$ (or $N-2n+1$ when the common line sum is allowed to vary freely, as in Ward \cite{ward1980vector}). By contrast, the two first-degree moment conditions $\langle x,z\rangle=\langle y,z\rangle=0$ leave an $(N-3)$-dimensional subspace in $\mathbf 1^\perp$.

\begin{definition}[Energies]
\label{def:energies}
For an arrangement $S$ with centred entries $z$,
\begin{align}
\Efull(S)&=\sum_{L}\langle\mathbf 1_L,z\rangle^{2}\qquad\text{(all $2n+2$ lines)},\label{eq:Efull}\\
\Efour(S)&=\langle x,z\rangle^{2}+\langle y,z\rangle^{2}+\langle\mathbf 1_{D_{\mathrm m}},z\rangle^{2}+\langle\mathbf 1_{D_{\mathrm a}},z\rangle^{2}.\label{eq:Elow}
\end{align}
We call $\Efour$ the \emph{four-term energy}.
\end{definition}

$\Efour$ keeps the diagonal conditions and replaces the $2n-2$ independent row and column conditions by the two degree-$1$ moment conditions (\Cref{prop:hierarchy}); it is $D_4$-invariant and has a fixed number of terms, which makes it a test case for a limit law (\Cref{cor:random_energy_clt}).
By \Cref{lem:variances},
\begin{equation}
N^{2}\Efull(S)=\sum_{L}\delta_L(S)^{2}=\sum_i\rho_i^{2}+\sum_j\gamma_j^{2}+\delta_{\mathrm m}^{2}+\delta_{\mathrm a}^{2},
\label{eq:Efull_devs}
\end{equation}
the sum of the squared deviations of all $2n+2$ lines, a non-negative integer.
This is the energy of the Monte Carlo study of Pinn and Wieczerkowski \cite{pinn1998number}, who use $\sum_L\delta_L^2=N^2\Efull$; we follow them in calling $\Efull$ and $\Efour$ energies.
$\Efull(S)=0$ if and only if $S$ is magic, and $\Efull$ depends on $z$ only through its projection onto $W$ (\Cref{fig:decomposition}).
Both energies are invariant under the dihedral group $D_4$ of the square, which permutes the lines and maps $\{\pm x,\pm y\}$ to itself.
All $n$ rows and $n$ columns are kept in \eqref{eq:Efull}: the last row and column are redundant for the zero set, but omitting them would destroy the $D_4$-invariance.

\begin{remark}[The four conditions]
\label{rem:four_term_insufficient}
$\Efour=0$ consists of four linear conditions on $z$, and ``magic'' of $2n$ independent ones; for every $n\ge3$ the four conditions are therefore strictly weaker than the magic conditions on real-valued $z$.
For permutations the two coincide when $n=3$ (\Cref{prop:n3_lowmode}) but not when $n=4$: the arrangement
\[
S=\begin{pmatrix}9&4&12&3\\16&1&14&7\\10&15&11&8\\2&5&6&13\end{pmatrix}
\]
has both diagonal sums $34$, row sums $(28,38,44,26)$ and column sums $(37,25,43,31)$, so $\gamma=(3,-9,9,-3)$ and $\rho=(-6,4,10,-8)$, and $\langle x,z\rangle=\langle y,z\rangle=0$; so $\Efour(S)=0$ while $N^2\Efull(S)=396$.
Both orders are treated in \Cref{sec:lowmode_n4}.
\end{remark}

\begin{remark}[Linear functionals and $\Efull$]
\label{rem:quant_balance}
By \Cref{lem:W_structure}, with $\delta=(\delta_{\mathrm m},\delta_{\mathrm a})$,
\[
\|P_Wz\|^2=\frac{n\bigl(\|\rho\|_2^2+\|\gamma\|_2^2\bigr)}{N^2}+\frac1N\,\delta^{\mathsf T}\Gamma_D^{-1}\delta ,
\]
and since the eigenvalues of $\Gamma_D$ are $n-1,n-1$ for odd $n$ and $n-2,n$ for even $n$, comparison with \eqref{eq:Efull_devs} gives $n\Efull\le\|P_Wz\|^2\le\lambda_n\Efull$, with $\lambda_n=n^2/(n-1)$ for odd $n$ and $\lambda_n=n^2/(n-2)$ for even $n$.
Hence $|\sum_cz_c\varphi_c|=N|\langle\varphi,P_Wz\rangle|\le N\|\varphi\|\sqrt{\lambda_n\Efull}$ for every $\varphi\in W$; for instance $x,y\in W_C\oplus W_R$, so the centre of mass $\sum_cz_c(x_c,y_c)$ is controlled by $\Efull$.
\end{remark}

\begin{figure}[htbp]
\centering
\includegraphics[width=\textwidth]{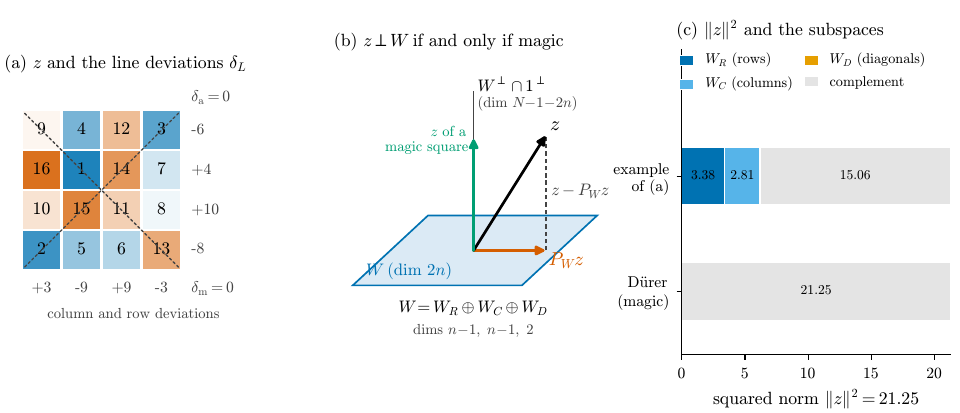}
\caption{The line space for $n=4$. $W$ is the $2n$-dimensional span of the centred line indicators, and $W^{\perp}\cap\mathbf 1^{\perp}$ (dimension $N-1-2n$) is its orthogonal complement among the zero-mean functions. (a)~The arrangement of \Cref{rem:four_term_insufficient}, shaded by $z$, with its row and column deviations (both diagonal deviations are $0$). (b)~Schematic of the orthogonal splitting $z=P_Wz+(z-P_Wz)$; the arrangement is magic if and only if $P_Wz=0$. (c)~The squared norms $\|P_{W_R}z\|^2,\|P_{W_C}z\|^2,\|P_{W_D}z\|^2$ and that of the complement, for the arrangement of (a) ($3.38+2.81+0+15.06$) and for D\"urer's magic square (all of $\|z\|^2=21.25$ lies in the complement).}
\label{fig:decomposition}
\end{figure}

\subsection{Permutation statistics}
\label{sec:bg_perm}

Let $\pi$ be a uniformly random permutation of $\{1,\dots,N\}$ and $a,a',b\in\R^N$.
For the linear permutation statistics $\Phi_a=\sum_ia_ib_{\pi(i)}$ one has $\E\Phi_a=N\bar a\bar b$ (here $\bar a=\frac1N\sum_ia_i$) and, by the theorem of Wald and Wolfowitz \cite{waldwolfowitz1944,noether1949} (polarised),
\begin{equation}
\Cov(\Phi_a,\Phi_{a'})=\frac{1}{N-1}\sum_i(a_i-\bar a)(a'_i-\bar a')\sum_j(b_j-\bar b)^2 .
\label{eq:ww_identity}
\end{equation}
For a uniformly random arrangement, $z_c=b_{\pi(c)}$ with $b$ the centred values $b_v=v-\frac{N+1}2$, $1\le v\le N$; we put
\begin{equation}
\sigma_n^2=\frac{\|z\|^2}{N-1}=\frac{N+1}{12}=\frac{n^2+1}{12}.
\label{eq:sigma}
\end{equation}
Hoeffding's combinatorial central limit theorem \cite[Thm.~4 and condition~(14), p.~560]{hoeffding1951} states that for real vectors $a=a^{(N)}$, $b=b^{(N)}$ the statistic $\sum_ia_ib_{\pi(i)}$ is asymptotically normal, i.e.\ $\bigl(\sum_ia_ib_{\pi(i)}-\E\bigr)/\sqrt{\Var}$ converges in distribution to $\mathcal N(0,1)$, if
\begin{equation}
\lim_{N\to\infty}N\,\frac{\max_i(a_i-\bar a)^2\,\max_j(b_j-\bar b)^2}{\sum_i(a_i-\bar a)^2\,\sum_j(b_j-\bar b)^2}=0 .
\label{eq:hoeffding_max}
\end{equation}
Motoo \cite{motoo1957} gives a Lindeberg-type condition and Bolthausen \cite{bolthausen1984} a Berry--Esseen bound.

\subsection{Known counts}
\label{sec:bg_polytopes}

There are $8$ magic squares of order $3$, all images of the Lo Shu under $D_4$; $7{,}040$ of order $4$, forming the $880$ $D_4$-orbits listed by Fr\'enicle de Bessy \cite{frenicle1693} and recounted many times since \cite{beck2010enumeration}; and $275{,}305{,}224$ $D_4$-orbits of order $5$, i.e.\ $2{,}202{,}441{,}792$ squares, counted by Schroeppel in 1973 and reported by Gardner \cite{gardner1976}.
For $n\ge6$ only statistical estimates are known \cite{pinn1998number,trump2012enumeration}.

\subsection{Notation}
\label{sec:notation}

\Cref{tab:notation} lists the main symbols and where they are defined.

\begin{table}[htbp]
\centering\small
\renewcommand{\arraystretch}{1.12}
\begin{tabular}{@{}>{\raggedright\arraybackslash}p{0.27\textwidth}>{\raggedright\arraybackslash}p{0.68\textwidth}@{}}
\toprule
symbol & meaning (where defined)\\
\midrule
$n$, $N=n^2$, $M(n)$ & order, number of cells, magic constant $n(N+1)/2$ (\Cref{def:magic_square})\\
$c=(i,j)$, $s_c$ & cell in row $i$ and column $j$ ($0\le i,j<n$); entry of the arrangement $S$\\
$x$, $y$, $z$ & centred column and row coordinates, centred entries (\Cref{def:coords})\\
$\langle f,g\rangle$, $\|f\|$, $\bar f$, $\tilde f$ & $N^{-1}\sum_cf_cg_c$, its norm, the mean and the centred version of $f$ (\Cref{sec:bg_coords})\\
$\delta_L$, $\rho$, $\gamma$, $\delta$ & deviation of the line $L$; row and column deviation vectors; $(\delta_{\mathrm m},\delta_{\mathrm a})$ \eqref{eq:line_dev}\\
$W$; $W_R$, $W_C$, $W_D$; $P_W$ & span of the centred line indicators; its row, column and diagonal parts; orthogonal projection onto $W$ (\Cref{lem:W_structure})\\
$\Gamma_D$, $\epsilon_n$ & $N$ times the Gram matrix of the centred diagonal indicators; $\epsilon_n=0$ ($n$ odd), $-1$ ($n$ even) (\Cref{lem:W_structure})\\
$\Efull$, $\Efour$ & line-sum energy and four-term energy (\Cref{def:energies})\\
$\sigma_n^2$ & $(n^2+1)/12$ \eqref{eq:sigma}\\
\midrule
$P_\gamma(t)$ & $\sum_j\gamma_jt^j$ (\Cref{sec:theory_hierarchy})\\
$\Lambda_{n,k}$ & $\{\gamma\in\Z^n:(t-1)^{k+1}\mid P_\gamma\}$ (\Cref{sec:kstar})\\
$k^*(n)$ & permutation threshold (\Cref{def:kstar})\\
$U(n)$ & least $k$ such that no nonzero vector of $\Lambda_{n,k}$ satisfies \eqref{eq:col_feasible} (\Cref{sec:kstar_exact})\\
\midrule
$T_\varphi$, $\tau_\ell$ & $\langle\varphi,z\rangle$ for a random arrangement; standard deviation of $T_\ell=T_{\varphi_\ell}$ (\Cref{sec:random_energy})\\
$\Gamma$, $C_n$ & Gram matrix of the centred test functions of an energy; correlation matrix of the four statistics of $\Efour$ (\Cref{prop:random_energy_moments}, \Cref{cor:random_energy_clt})\\
$A$, $h_c$ & matrix of an energy as a quadratic form in $z$; $h_c=A_{cc}$ (\Cref{sec:random_var}, Appendix~\ref{sec:app_variance})\\
$L_n$, $U_1,\dots,U_4$, $\Sigma_U$ & arrangements with $\Efour=0$; the four additive statistics defining $L_n$; their covariance matrix (\Cref{sec:lowmode_n4})\\
\bottomrule
\end{tabular}
\caption{Notation.}
\label{tab:notation}
\end{table}

\section{Moment conditions and the permutation threshold}
\label{sec:theory}

\subsection{Moment conditions}
\label{sec:theory_hierarchy}

For a centred $f\in\mathcal F$ let $\gamma_j(f)=\sum_if_{ij}$ and $\rho_i(f)=\sum_jf_{ij}$ be its column and row sums; for $f=z(S)$ these are the deviations $\gamma,\rho$ of $S$.
For $\gamma\in\R^n$ put $P_\gamma(t)=\sum_{j=0}^{n-1}\gamma_jt^j$.

\begin{definition}[Degree-$k$ conditions]
\label{def:degree_k}
A centred $f\in\mathcal F$ satisfies the \emph{degree-$k$ conditions} ($k\ge1$) if $\langle x^{m},f\rangle=\langle y^{m},f\rangle=0$ for $1\le m\le k$.
For $f=z(S)$ these read $\sum_{i,j}z_{ij}x_j^m=\sum_{i,j}z_{ij}y_i^m=0$.
\end{definition}

\begin{lemma}[Polynomial form]
\label{lem:poly_form}
If $\varphi$ depends only on the column index, $\langle\varphi,f\rangle=\frac1N\sum_j\varphi(j)\gamma_j(f)$; similarly for rows.
For $\gamma\in\R^{n}$ and $k\ge0$ the following are equivalent:
(i)~$\sum_jx_j^{m}\gamma_j=0$ for $0\le m\le k$;
(ii)~$\sum_j\binom{j}{r}\gamma_j=0$ for $0\le r\le k$;
(iii)~$(t-1)^{k+1}$ divides $P_\gamma(t)$.
Hence a centred $f$ satisfies the degree-$k$ conditions if and only if $(t-1)^{k+1}$ divides both $P_{\gamma(f)}$ and $P_{\rho(f)}$.
\end{lemma}

\begin{proof}
The first statement is $\sum_{i,j}\varphi(j)f_{ij}=\sum_j\varphi(j)\gamma_j(f)$.
The families $(x_j^{m})_{m\le k}$ and $(\binom{j}{r})_{r\le k}$ both span the restrictions to $\{0,\dots,n-1\}$ of the polynomials of degree at most $k$ in $j$, so (i)$\Leftrightarrow$(ii); and $P_\gamma(t)=\sum_{r}\bigl(\sum_j\binom jr\gamma_j\bigr)(t-1)^{r}$ is the Taylor expansion at $t=1$, so (ii)$\Leftrightarrow$(iii).
The case $m=0$ of (i) holds for every centred $f$.
\end{proof}

\begin{proposition}[Moment conditions]
\label{prop:hierarchy}
Let $n\ge3$ and let $f\in\mathcal F$ be centred.
\begin{enumerate}
\item[(a)] If $\gamma(f)=\rho(f)=0$, then $\langle a(x)+b(y),f\rangle=0$ for all functions $a,b$ on $\{x_0,\dots,x_{n-1}\}$; in particular $f$ satisfies the degree-$k$ conditions for every $k$.
For $f=z(S)$ with $S$ magic, also $\langle a(x)+b(y)+c_1\mathbf 1_{D_{\mathrm m}}+c_2\mathbf 1_{D_{\mathrm a}},z\rangle=0$.
\item[(b)] If $k\ge n-1$, the degree-$k$ conditions imply $\gamma(f)=\rho(f)=0$.
\item[(c)] If $1\le k\le n-2$, the function $f_{ij}=\frac1n(-1)^{j}\binom{k+1}{j}$ ($0\le j\le k+1$), $f_{ij}=0$ ($j>k+1$), satisfies the degree-$k$ conditions but not the degree-$(k+1)$ conditions, and $\gamma(f)\ne0$.
\end{enumerate}
\end{proposition}

\begin{proof}
(a) is the first statement of \Cref{lem:poly_form}.  (b) A nonzero polynomial of degree at most $n-1$ is not divisible by $(t-1)^{n}$.  (c) Here $\rho(f)=0$ and $P_{\gamma(f)}(t)=(1-t)^{k+1}$.
\end{proof}

Part (a) contains the vanishing of the centre of mass $\sum_cz_c(x_c,y_c)$ and of all axial moments $\sum_cz_cx_c^m$, $\sum_cz_cy_c^m$ of a semi-magic square, in the spirit of the multipole analysis of a magic cube by Rogers and Loly \cite{rogers2005multipole}.
Parts (b) and (c) are the invertibility of a Vandermonde matrix.
The mixed moments are not determined by the line sums:

\begin{remark}[Mixed moments]
\label{rem:product_inertia}
The magic square with rows $(1,6,11,16)$, $(8,15,2,9)$, $(12,3,14,5)$, $(13,10,7,4)$ (square $\#45$ in the Fr\'enicle numbering of \cite{mathewson_gems}) has mixed moment $\sum_cz_cx_cy_c=60$ (Loly's product of inertia).
Among the $880$ magic squares of order $4$ up to $D_4$, the mixed moment $\sum_cz_cx_cy_c$ is nonzero for $824$ and takes $48$ distinct absolute values between $0$ and $60$, and the mixed moments $\sum_cz_cx_c^2y_c$, $\sum_cz_cx_cy_c^2$ are nonzero for $623$ and $773$ of them.
For $n=3$ the mixed moment $\sum_cz_cx_cy_c$ vanishes for all eight magic squares, because $s_c+s_{-c}=10$ for centrally opposite cells, so $z$ is odd under $(x,y)\mapsto(-x,-y)$ while $xy$ is even.
\end{remark}

\begin{corollary}[Loly's invariant {\cite{loly2004invariance}}]
\label{cor:loly}
Let $I(S)=\sum_{i,j}s_{ij}(x_{j}^{2}+y_{i}^{2})$ be the moment of inertia about the centre of the masses $s_{ij}$ placed at $(x_j,y_i)$.  For every arrangement,
\begin{equation}
I(S)=N\,\langle x^{2}+y^{2},z\rangle+\frac{n^{2}(n^{4}-1)}{12}
=\sum_j x_j^{2}\gamma_j+\sum_i y_i^{2}\rho_i+\frac{n^{2}(n^{4}-1)}{12}.
\label{eq:loly_identity}
\end{equation}
Hence $I(S)=n^{2}(n^{4}-1)/12$ if and only if the single condition $\langle x^2+y^2,z\rangle=0$ holds; it is implied by the degree-$2$ conditions, hence by the semi-magic property.
\end{corollary}

\begin{proof}
Substitute $s_{ij}=z_{ij}+\frac{N+1}2$ and use \Cref{lem:variances} and \Cref{lem:poly_form}: $\frac{N+1}2\sum_{i,j}(x_j^{2}+y_i^{2})=\frac{N+1}2\cdot\frac{n^{2}(n^{2}-1)}{6}$.
\end{proof}

Loly's derivation uses the row and column sums in the same way \cite{loly2004invariance}.
The single condition is much weaker than the degree-$2$ conditions: for $n=3$, $24{,}192$ of the $9!$ arrangements have $I=60$, while $72$ (the semi-magic squares, \Cref{rem:n3_enum}) satisfy the degree-$2$ conditions.
The energy $\Efour$ of \eqref{eq:Elow} combines the degree-$1$ conditions with the two diagonal conditions: $\Efour=\langle x,z\rangle^2+\langle y,z\rangle^2+(\delta_{\mathrm m}^2+\delta_{\mathrm a}^2)/N^2$.

\subsection{The permutation threshold}
\label{sec:kstar}

By \Cref{prop:hierarchy}, for real-valued $f$ the degree-$k$ conditions force equal row and column sums exactly when $k\ge n-1$.
For arrangements the deviations are integers of bounded size, and the threshold drops.

\begin{definition}
\label{def:kstar}
For $n\ge3$, $k^*(n)$ is the least $k\ge1$ such that every arrangement of order $n$ satisfying the degree-$k$ conditions is semi-magic.
A \emph{witness} for $k^*(n)\ge K$ is an arrangement that satisfies the degree-$(K-1)$ conditions and is not semi-magic.
\end{definition}

By \Cref{prop:hierarchy}(b), $k^*(n)\le n-1$.
By \Cref{lem:poly_form}, an arrangement satisfies the degree-$k$ conditions if and only if $\gamma$ and $\rho$ lie in
\[
\Lambda_{n,k}=\{\gamma\in\Z^{n}:(t-1)^{k+1}\mid P_\gamma(t)\}=(t-1)^{k+1}\,\Z[t]_{<n-k-1}
\]
(identifying $\gamma$ with $P_\gamma$; the quotient of an integer polynomial by the monic polynomial $(t-1)^{k+1}$ has integer coefficients), a lattice of rank $n-k-1$.
Writing $\gamma=\gamma^+-\gamma^-$ with $\gamma^\pm\ge0$, the condition $\gamma\in\Lambda_{n,k}$ says that the multisets in which $j$ occurs $\gamma^+_j$, resp.\ $\gamma^-_j$, times have equal power sums of degrees $0,\dots,k$: a Prouhet--Tarry--Escott system \cite{wright1959prouhet,borwein2002computational}.
Small vectors in $\Lambda_{n,k}$ are classical: the Prouhet--Thue--Morse vector $\bigl((-1)^{s_2(j)}\bigr)_{0\le j<2^{k+1}}$ ($s_2$ the binary digit sum) has entries $\pm1$ and lies in $\Lambda_{2^{k+1},k}$ \cite{wright1959prouhet,allouche1999ubiquitous}, and Borwein, Erd\'elyi and K\'os construct polynomials of degree $n$ with coefficients in $\{-1,0,1\}$ vanishing to order $\ge c\sqrt{n/\log n}$ at $1$ \cite[Thm.~2.7]{BEK1999}; Borwein and Mossinghoff determine the least degree of such a polynomial with a zero of order $m$ at $1$ for $m\le10$ \cite{borwein2000height}.

\begin{remark}[Prouhet--Thue--Morse sign patterns versus polynomial collisions]
\label{rem:ptm_scale}
The explicit Prouhet--Thue--Morse vector provides an exact $\pm 1$ solution to the Prouhet--Tarry--Escott system, but its recursive block-doubling structure imposes an exponential length requirement: achieving degree $k$ demands $2^{k+1}$ coordinates. Confined to $n$ columns, this algebraic rigidity restricts the vanishing degree to $2^{k+1}\le n$, yielding only $k=\Theta(\log n)$. In sharp contrast, non-constructive pigeonhole collisions in the polynomial ring $\Z[t]$ modulo $(t-1)^{k+1}$ (as employed in the proof of \Cref{thm:perm_threshold}(b) below via Siegel-type box principles) exploit the combinatorial volume of the polynomial ring, unlocking non-trivial collisions at the vastly larger $\sqrt n$ scale.
\end{remark}

The question is therefore how large $k$ can be for vectors that are column-deviation vectors of arrangements.

\begin{theorem}[The permutation threshold]
\label{thm:perm_threshold}
\leavevmode
\begin{enumerate}
\item[(a)] For all $n\ge3$, $\displaystyle k^{*}(n)\le\Bigl\lfloor\sqrt{2(n-1)\log(n^{4}/4)}\Bigr\rfloor\le 2\sqrt2\,\sqrt{n\log n}$.
\item[(b)] $\displaystyle\liminf_{n\to\infty}\frac{k^{*}(n)}{\sqrt n}\ge\sqrt2$, and $\displaystyle\liminf_{n\to\infty,\ 4\mid n}\frac{k^{*}(n)}{\sqrt n}\ge\sqrt6$.
\item[(c)] $k^{*}(n)=n-1$ for $3\le n\le13$, $k^{*}(n)=n-2$ for $14\le n\le17$, and $k^{*}(n)\le n-3$ for every $n\ge18$.
\item[(d)] \textup{(Computer-assisted.)} For $18\le n\le40$, $k^*(n)$ is the value given in \Cref{tab:kstar}.
\end{enumerate}
\end{theorem}

Part (a) is Schur's bound on real zeros applied to $P_\gamma$ (\Cref{lem:schur}) and does not use that $\gamma$ comes from an arrangement beyond the bound $\|\gamma\|_1\le n^4/4$ of \Cref{lem:col_feasible}; the constant $2\sqrt2$ is about $2.83$.
It is below $n-1$ from $n=24$ on, gives $k^*(n)\le n-3$ from $n=27$ on and improves on (c) from $n=30$ on; at $n=40$ and $n=100$ it gives $32$ and $58$, against the exact value $26$ and the computed upper bound $45$ (\Cref{fig:kstar}).
In part (c) the lower bounds rest on witness arrangements checked in exact arithmetic (printed in Appendix~\ref{sec:app_witnesses} for $n\le13$ and documented in Appendix~\ref{sec:appendix} for $14\le n\le17$), while the upper bounds are proved by hand.
Part (d) is computer-assisted: its upper bounds come from an exact lattice enumeration backed by an arithmetic certificate, and its lower bounds from witness arrangements found by computer search and checked in exact arithmetic (\Cref{sec:kstar_exact}).

\subsection{Proof of the upper bound}
\label{sec:kstar_upper}

\begin{lemma}[Majorisation bound for column deviations]
\label{lem:col_feasible}
Let $\gamma$ be the column-deviation vector of an arrangement.  For every set $J$ of $p$ columns,
\begin{equation}
\Bigl|\sum_{j\in J}\gamma_j\Bigr|\le \frac{pn(N-pn)}{2}.
\label{eq:col_feasible}
\end{equation}
In particular $|\gamma_j|\le n^{2}(n-1)/2$, $\|\gamma\|_1\le N^{2}/4$ and $\|\gamma\|_2^2\le n^5(n^2-1)/12$.  The same holds for $\rho$.
\end{lemma}

\begin{proof}
The $pn$ cells of $J$ carry distinct values, whose sum is at most that of the $pn$ largest values, $pn(2N-pn+1)/2$; subtracting $pM(n)=pn(N+1)/2$ gives the upper bound, and the lower bound is symmetric.
With $p$ the number of positive $\gamma_j$, $\|\gamma\|_1=2\sum_{\gamma_j>0}\gamma_j\le pn(N-pn)\le N^2/4$.
Finally, $\gamma$ is majorised by the vector $v$ with $v_p=n(N-(2p-1)n)/2$, $1\le p\le n$ (equal totals, $\sum_pv_p=0=\sum_j\gamma_j$, and \eqref{eq:col_feasible} bounds the partial sums of $\gamma$ sorted in decreasing order by $\sum_{p\le P}v_p=Pn(N-Pn)/2$); $\|v\|_2^2=n^5(n^2-1)/12$, and $\|\cdot\|_2^2$ is Schur-convex.
\end{proof}

\begin{lemma}[Schur's bound at $t=1$]
\label{lem:schur}
Let $\gamma\in\Z^{n}\setminus\{0\}$ with $(t-1)^{K}\mid P_\gamma$ for some $K\ge1$.  Then $K^{2}\le 2(n-1)\log\|\gamma\|_1$.
\end{lemma}

\begin{proof}
By a theorem of Schur \cite{schur1933} (see \cite[(1.4)]{erdosturan1950}), a polynomial $f=\sum_{j=0}^ma_jt^j$ with $a_0a_m\ne0$ has at most $\bigl(4m\log\bigl(\sum_j|a_j|/\sqrt{|a_0a_m|}\bigr)\bigr)^{1/2}$ real zeros, counted with multiplicity.
Write $P_\gamma=t^{\nu}Q$ with $Q(0)\ne0$, and let $d=\deg Q\le n-1$.
The nonzero coefficients of $Q$ are the nonzero entries of $\gamma$, so the product of its constant and leading coefficients is a nonzero integer, the sum of their absolute values is $\|\gamma\|_1$, and $(t-1)^K\mid Q$.
The polynomial $Q(t^2)$ has degree $2d$, the same nonzero coefficients, and is divisible by $(t^2-1)^K$, so it has the real zeros $\pm1$, each of multiplicity at least $K$.
Hence $(2K)^2\le8d\log\|\gamma\|_1$.
\end{proof}

The same argument gives the form of Schur's theorem for positive zeros in \cite[Thm.~A]{BEK1999}.
Mignotte's lower bound $h\ge K^2/(4D)-\log D$ for the logarithmic height $h$ of a nonzero integer polynomial of degree less than $D$ with a zero of order $K$ at $1$ \cite{mignotte1982}, quoted in \cite[Sect.~1]{amoroso1995}, rests on the same theorem; with $h\le\log\max_j|\gamma_j|$ and \Cref{lem:col_feasible} it gives $K\le(4+o(1))\sqrt{n\log n}$.
The order $\sqrt{n\log n}$ also follows from \cite[Thms.~2.1 and 2.8]{BEK1999}, applied to $P_\gamma(t)/(t^{\nu}\max_j|\gamma_j|)$.

\begin{proof}[Proof of \Cref{thm:perm_threshold}(a)]
Let $S$ satisfy the degree-$k$ conditions with $(k+1)^2>2(n-1)\log(n^4/4)$.
If $\gamma\ne0$, then $(t-1)^{k+1}\mid P_\gamma$ (\Cref{lem:poly_form}) and $\|\gamma\|_1\le n^4/4$ (\Cref{lem:col_feasible}), contradicting \Cref{lem:schur}.
So $\gamma=0$, and likewise $\rho=0$.
The second inequality follows from $(n-1)\log(n^4/4)<4n\log n$.
\end{proof}

\Cref{lem:schur} uses only that $\gamma$ is a nonzero integer vector with $\|\gamma\|_1\le n^4/4$, and \Cref{lem:col_feasible} does not restrict the order of vanishing much further.
By Amoroso's upper bound \cite[Thm.~2]{amoroso1995}, obtained from the Bombieri--Vaaler form of Siegel's lemma \cite{bombieri1987low}, as $D,K\to\infty$ with $K/D\to0$ and $K\sqrt{\log D/D}\to\infty$ there are nonzero integer polynomials of degree less than $D$ with a zero of order $K$ at $1$ whose coefficients are at most $\exp\bigl((\frac14+o(1))\frac{K^2}D\log\frac DK\bigr)$ in absolute value.\footnote{Amoroso \cite[Thm.~2]{amoroso1995} obtains the constant $\frac14$ in the exponent $\exp\bigl((\frac14+o(1))\frac{K^2}D\log\frac DK\bigr)$ by forcing the polynomial to vanish at cyclotomic points of small degree. However, the auxiliary construction in his proof operates in the asymptotic regime $\sqrt{D\log D}/K\to0$, which does not hold at the critical scaling $K\asymp\sqrt D$ where $\sqrt{D\log D}/K\asymp\sqrt{\log D}\to\infty$. While Amoroso's result motivates the potential order $\sqrt{24n}$, applying Dirichlet's box principle directly to $\Lambda_{n,k}$ unconditionally establishes the $(\sqrt{12}-o(1))\sqrt n$ bound of \Cref{thm:perm_threshold}(b) without relying on cyclotomic vanishing or on \cite{amoroso1995}.}
For $D=n$ and $K=\lfloor\alpha\sqrt n\rfloor$ with $\alpha<\sqrt{24}$ this is $n^{\alpha^2/8+o(1)}<n^3/4$ for large $n$; then $\|\gamma\|_1\le n^4/4$, and $|\sum_{j\in J}\gamma_j|\le pn^3/4\le pn(N-pn)/2$ for $|J|=p\le n/2$, and also for $p>n/2$ since $\sum_j\gamma_j=0$.
So for large $n$ the lattice $\Lambda_{n,k}$ with $k=(\sqrt{24}-o(1))\sqrt n$ contains nonzero vectors satisfying \eqref{eq:col_feasible}, and an upper bound for $k^*(n)$ below this order must use more than \Cref{lem:col_feasible}.
The pigeonhole principle alone, as in the proof of (b) below, gives the constant $\sqrt{12}$ in place of $\sqrt{24}$.
Between $\sqrt{24n}$ and \Cref{lem:schur} lies the gap, by a factor of order $\log(D/K)$ in the height, between the known lower and upper bounds for the least height of an integer polynomial of degree less than $D$ with a zero of order $K$ at $1$ \cite{mignotte1982,bombieri1987low,amoroso1995}; for coefficients in $\{-1,0,1\}$ the analogous gap is between the orders $\sqrt{n/\log n}$ and $\sqrt n$, and closing it ``looks quite difficult'' \cite[Sect.~2]{BEK1999}.
Arguments that use integrality only through the size of one coefficient cannot close it: rescaled, \cite[Thm.~2.3]{BEK1999} gives real polynomials of degree $n$ with constant coefficient $1$ and all coefficients at most $H\ge1$ in absolute value that have a zero of order at least $\min\{\frac16\sqrt{n(1+\log H)}-1,n\}$ at $1$; see \cite[Thms.~2.1 and 2.2]{erdelyi2016} for explicit constants in the $\ell^p$ normalisations.

\subsection{Proof of the lower bound: complementary-row swaps and a pigeonhole collision}
\label{sec:kstar_lower}

To prove a lower bound $k^*(n)\ge K$, we must construct a witness arrangement whose row sums all equal $M(n)$ (so $\rho=0$), whose column-deviation vector $\gamma$ is nonzero, and whose polynomial $P_\gamma(t)=\sum_j\gamma_jt^j$ is divisible by $(t-1)^K$. The difficulty is that $\gamma$ cannot be chosen freely in the lattice $\Lambda_{n,K-1}$: it must arise from an actual permutation of $\{1,\dots,n^2\}$ with balanced rows.

Our strategy is to start from a semi-magic square $S$ and perturb it by horizontal transpositions—swapping two entries within the same row—which automatically preserve every row sum. Specifically, we group the columns into $Q\approx n/2$ adjacent pairs $(j_q, j_q+1)$ with $j_q=j_0+2q$. Any horizontal swap between columns $j_q$ and $j_q+1$ changes the column-deviation vector by a multiple of $e_{j_q}-e_{j_q+1}$, which contributes a dipole term $\zeta_q t^{j_q}(1-t)$ to $P_\gamma(t)$. Summing over all $q<Q$, the resulting polynomial factors as
\[
P_\gamma(t)=t^{j_0}(1-t)P_\zeta(t^2),\qquad\text{where}\quad P_\zeta(u)=\sum_{q<Q}\zeta_q u^q.
\]
Thus, if we can choose the net swap amplitudes $\zeta_q$ so that the first $K-1$ binomial moments $B_s(\zeta)=\sum_q\zeta_q\binom qs$ ($0\le s\le K-2$) all vanish, then $(u-1)^{K-1}\mid P_\zeta(u)$, which implies $(t^2-1)^{K-1}\mid P_\zeta(t^2)$ and hence $(t-1)^K\mid P_\gamma(t)$.

How do we find a nonzero vector $\zeta$ of achievable swap amplitudes with $K-1$ vanishing moments? By the pigeonhole principle (in the spirit of Siegel's lemma): in each column pair $q$, we select $r_q$ rows and consider all subset sums $g_q \in F_q$ of their swap differences $d_{q,i}=s_{i,j_q+1}-s_{i,j_q}$. If the box $\prod_{q<Q} F_q$ has more elements than the number of possible moment tuples $(B_0(g),\dots,B_{K-2}(g))$, two distinct subset-sum vectors $g\ne g'$ must share the exact same moments, so their difference $\zeta = g - g'$ has $B_s(\zeta)=0$ for all $s\le K-2$.

There is just one catch: $\zeta_q = g_q - g'_q$ is a \emph{difference} of two subset sums, meaning we must apply $+d_{q,i}$ for $i\in I_q\setminus I'_q$ and $-d_{q,i}$ for $i\in I'_q\setminus I_q$ simultaneously in disjoint rows. This is why we build our starting semi-magic squares out of \emph{complementary pairs of rows}: if row $i$ and row $i'$ sum to $N+1$ in every column, then the entry difference in row $i'$ is the exact negative of the entry difference in row $i$, giving us both $+d_{q,i}$ and $-d_{q,i}$ on disjoint cells.

\begin{definition}
Rows $i\ne i'$ of an arrangement are \emph{complementary on} a set $J$ of columns if $s_{ij}+s_{i'j}=N+1$ for all $j\in J$.
\end{definition}

\begin{lemma}[Swaps in complementary rows]
\label{lem:twin}
Let $S$ be semi-magic, let rows $i,i'$ be complementary on $J$, let $j,j'\in J$ and $d=s_{ij'}-s_{ij}$.
Exchanging the entries of row $i$ in columns $j,j'$ changes the column deviations by $d(e_j-e_{j'})$, exchanging those of row $i'$ changes them by $-d(e_j-e_{j'})$, and neither changes a row sum.
Exchanges involving pairwise disjoint sets of cells can be combined, and their effects add.
\end{lemma}

\begin{proof}
$s_{i'j}-s_{i'j'}=(N+1-s_{ij})-(N+1-s_{ij'})=d$.
\end{proof}

\begin{lemma}[Squares with many complementary rows]
\label{lem:twin_squares}
\leavevmode
\begin{enumerate}
\item If $n=2h$ with $h$ even, or with $h$ odd and $h\ge7$, there is a semi-magic square of order $n$ in which rows $r$ and $h+r$ are complementary on all columns, for $0\le r<h$.
\item If $n\ge9$ is odd, there is a semi-magic square of order $n$ in which rows $\ell$ and $n-1-\ell$ are complementary on the columns $\ell+1,\dots,n-2-\ell$, for $0\le\ell\le(n-9)/2$.
\end{enumerate}
\end{lemma}

Part (1) does not cover $n=6,10$; this is irrelevant for \Cref{thm:perm_threshold}(b), which concerns large $n$.
The proof is a direct construction, given in Appendix~\ref{sec:app_twin}.

\begin{proposition}[Pigeonhole criterion]
\label{prop:box}
Let $S$ be a semi-magic square of order $n$, let $Q\ge1$ and $j_0\ge0$ with $j_0+2Q\le n$, and put $j_q=j_0+2q$ for $q<Q$.
For each $q<Q$ let $r_q$ pairwise disjoint pairs of rows be complementary on the columns $j_q,j_q+1$, let $d_{q,1},\dots,d_{q,r_q}$ be the differences $s_{i,j_q+1}-s_{i,j_q}$ for the first rows $i$ of these pairs, and let $F_q$ be the set of subset sums $\sum_{i\in I}d_{q,i}$, $I\subseteq\{1,\dots,r_q\}$.
If $K\ge2$ and
\begin{equation}
\prod_{q<Q}|F_q|\ >\ \prod_{s=0}^{K-2}\Bigl((N-1)\sum_{q<Q}r_q\binom qs+1\Bigr),
\label{eq:box}
\end{equation}
then there is an arrangement with all row sums $M(n)$ and $\gamma\ne0$ that satisfies the degree-$(K-1)$ conditions; hence $k^*(n)\ge K$.
\end{proposition}

\begin{proof}
The elements of $F_q$ lie in an interval of length $\sum_i|d_{q,i}|\le r_q(N-1)$.
For $g\in\prod_qF_q$ the binomial moments $B_s(g)=\sum_qg_q\binom qs$, $0\le s\le K-2$, are integers in intervals of lengths at most $(N-1)\sum_qr_q\binom qs$, so by \eqref{eq:box} there are $g\ne g'$ with $B_s(g)=B_s(g')$ for all $s\le K-2$.
Write $g_q=\sum_{i\in I_q}d_{q,i}$ and $g'_q=\sum_{i\in I'_q}d_{q,i}$.
For distinct column pairs $q\ne q'$, the column sets $\{j_q, j_q+1\}=\{j_0+2q, j_0+2q+1\}$ and $\{j_{q'}, j_{q'}+1\}$ are disjoint because the index advances by $2$.
Within each column pair $q$, the $r_q$ pairs of complementary rows are pairwise disjoint by hypothesis.
For every $q$, we perform the exchange of \Cref{lem:twin} in the first row of each pair in $I_q\setminus I'_q$ and in the second row of each pair in $I'_q\setminus I_q$.
Because the symmetric differences $I_q\setminus I'_q$ and $I'_q\setminus I_q$ are disjoint (omitting any common indices in $I_q\cap I'_q$), and the rows in each complementary pair are distinct, every cell $(i,j)$ in the entire grid is swapped at most once, so all exchanges involve pairwise disjoint cells.
The resulting arrangement retains all row sums $M(n)$ and has column deviations $\gamma=\sum_{q<Q}\zeta_q(e_{j_q}-e_{j_q+1})$ with $\zeta=g-g'\ne0$.
Its generating polynomial expands explicitly as
\begin{equation}
\sum_{j=0}^{n-1}\gamma_jt^j=\sum_{q<Q}\zeta_q(t^{j_q}-t^{j_q+1})=\sum_{q<Q}\zeta_qt^{j_0+2q}(1-t)=t^{j_0}(1-t)\sum_{q<Q}\zeta_q(t^2)^q=t^{j_0}(1-t)P_\zeta(t^2),
\label{eq:box_poly_step}
\end{equation}
where $P_\zeta(u)=\sum_{q<Q}\zeta_qu^q\ne0$.
Since $B_s(\zeta)=0$ for $s\le K-2$, \Cref{lem:poly_form} gives $(u-1)^{K-1}\mid P_\zeta(u)$, so $(t^2-1)^{K-1}\mid P_\zeta(t^2)$ and $(t-1)^{K}\mid P_\gamma$.
With $\rho=0$, the arrangement satisfies the degree-$(K-1)$ conditions, and $\gamma\ne0$.
\end{proof}

The $d_{q,i}$ are nonzero, and at least $\lceil r_q/2\rceil$ of them have the same sign; their partial sums, together with $0$, are distinct, so $|F_q|\ge\lceil r_q/2\rceil+1$.
For the squares of \Cref{lem:twin_squares} we take $j_0=0$, $Q=n/2$ and $r_q=n/2$ for even $n$, and $j_0=1$, $Q=(n-3)/2$ and $r_q=\min(2q,\,n-4-2q,\,(n-9)/2)+1$ for odd $n$: the $q$-th pair then consists of the columns $2q+1,2q+2$, and the complementary rows $\ell,n-1-\ell$ of \Cref{lem:twin_squares}(2) cover both columns if and only if $\ell\le2q$ and $\ell\le n-4-2q$.

\begin{proof}[Proof of \Cref{thm:perm_threshold}(b), first part]
Fix $\alpha<\sqrt2$ and let $K-1=k=\lfloor\alpha\sqrt n\rfloor$; we show that \eqref{eq:box} holds, with $|F_q|$ replaced by $\lceil r_q/2\rceil+1$, for all large $n$ covered by \Cref{lem:twin_squares} (every $n\ge14$ is covered).

\emph{Right side.} $\sum_{q<Q}\binom qs=\binom Q{s+1}\le\binom n{s+1}$, $r_q\le n$ and $N\le n^2$, so each factor is at most $n^4\binom n{s+1}\le n^4\,n^{s+1}/(s+1)!$.
From $\log s!\ge s\log s-s$ and $\sum_{s\le k}s\log s\ge\int_0^kx\log x\,dx=\frac{k^2}2\log k-\frac{k^2}4$,
\[
\sum_{s=1}^{k}\log s!\ \ge\ \frac{k^2}2\log k-\frac34k^2-\frac k2 ,
\]
hence $\sum_{s=1}^{k}(s\log n-\log s!)\le\frac{k^2}2\log\frac nk+O(k^2+k\log n)$ and
\[
\log(\text{right side})\le4k\log n+\sum_{s=1}^{k}\bigl(s\log n-\log s!\bigr)\le\frac{k^2}2\log\frac nk+O(k^2+k\log n)=\frac{\alpha^2}4\,n\log n+O(n).
\]

\emph{Left side.} For even $n$ it is at least $(\lceil n/4\rceil+1)^{n/2}$, with logarithm $\frac n2\log n-O(n)$.
For odd $n$, with $Q=(n-3)/2$, $r_q=\min(2q+1,\,2(Q-q),\,Q-2)\ge\min(q+1,\,Q-q)$ once $Q\ge5$ ($n\ge13$), so $\lceil r_q/2\rceil+1\ge\min(q+1,Q-q)/2$.
Hence
\[
\log(\text{left side})\ \ge\ \sum_{q<Q}\log\min(q+1,Q-q)-Q\log2\ \ge\ 2\log\bigl(\lfloor Q/2\rfloor!\bigr)-Q\log2=\frac n2\log n-O(n).
\]
Since $\alpha^2/4<1/2$, \eqref{eq:box} holds for large $n$, and $k^*(n)\ge\lfloor\alpha\sqrt n\rfloor+1$.
\end{proof}

The odd case thus gives the same constant $\sqrt2$ as the even case, but more slowly: \eqref{eq:box} with $|F_q|\ge\lceil r_q/2\rceil+1$ first gives $k^*(n)\ge2$ at $n=14$ for even $n$ and only at $n=21$ for odd $n$, and $k^*(n)\ge29$ at $n=999$ and $n=1001$ against $k^*(1000)\ge32$.
Notice why the general bound above stops at $\sqrt{2n}$: it uses only the generic estimate $|F_q|\ge\lceil r_q/2\rceil+1=\Theta(n)$, which gives $\log\prod_{q<Q}|F_q|=\frac12n\log n-O(n)$. When $n\equiv0\pmod 4$, we can construct a semi-magic square whose swap differences $d_{q,r}$ in column pair $q$ form a shifted arithmetic progression $\mp2(2qh+2r+1)$, expanding the subset-sum set $|F_q|$ from $\Theta(n)$ to $\Theta(qn^2)$ and tripling $\log\prod_{q<Q}|F_q|$ to $\frac32n\log n-O(n)$, which boosts the threshold from $\sqrt{2n}$ to $\sqrt{6n}$.

\begin{proposition}[All subset sums]
\label{prop:subset}
Let $n=2h$ with $h\ge4$ even, and for $0\le q,r<h$ put $\theta=qh+r$ and $\epsilon_q=(-1)^q$.
\begin{enumerate}
\item[(i)] The square with $s_{r,2q}=\frac12\bigl(N+1+\epsilon_q(4\theta+1)\bigr)$, $s_{r,2q+1}=\frac12\bigl(N+1-\epsilon_q(4\theta+3)\bigr)$ and $s_{h+r,j}=N+1-s_{r,j}$ is semi-magic, and its rows $r$ and $h+r$ are complementary on all columns.
\item[(ii)] With $j_0=0$, $Q=h$ and the $r_q=h$ pairs of rows $(r,h+r)$, the sets $F_q$ of \Cref{prop:box} satisfy $|F_q|\ge qh^2/40$ for $1\le q<h$.
\item[(iii)] Hence $k^*(n)\ge(\sqrt6-o(1))\sqrt n$ as $n\to\infty$ with $4\mid n$.
\end{enumerate}
\end{proposition}

\begin{proof}
(i) As $\theta$ runs over $0,\dots,h^2-1$, the numbers $4\theta+1$ and $4\theta+3$ run over the odd numbers $1,\dots,N-1$, so the entries $\frac12(N+1\pm(4\theta+1))$, $\frac12(N+1\pm(4\theta+3))$ of rows $r$ and $h+r$ are $1,\dots,N$, each once.
In row $r$ the two entries in columns $2q,2q+1$ add up to $N+1-\epsilon_q$, so the row sum is $h(N+1)=M(n)$ since $h$ is even; the same holds for row $h+r$.
Each column contains $h$ pairs of entries adding up to $N+1$.

(ii) Here $d_{q,r}=s_{r,2q+1}-s_{r,2q}=-2\epsilon_q(2qh+2r+1)$, so $|F_q|$ is the number of subset sums of $A_q=\{2qh+2r+1:0\le r<h\}$.
For $0\le\kappa\le h$ the sums of the $\kappa$-element subsets of $\{0,\dots,h-1\}$ are all integers in $[\frac{\kappa(\kappa-1)}2,\frac{\kappa(\kappa-1)}2+\kappa(h-\kappa)]$ (pass from $\{0,\dots,\kappa-1\}$ to $\{h-\kappa,\dots,h-1\}$ by raising one element by $1$ at a time), so the sums of the $\kappa$-element subsets of $A_q$ form a set $E_\kappa$ of $\kappa(h-\kappa)+1$ integers in $[2qh\kappa+\kappa^2,\,2qh\kappa+\kappa^2+2\kappa(h-\kappa)]$.
Let $\lambda=\lceil h/(4q)\rceil$.
If $\kappa'\ge\kappa+\lambda$, then $\min E_{\kappa'}\ge2qh(\kappa+\lambda)+(\kappa+\lambda)^2>2qh\kappa+\frac{h^2}2+\kappa^2\ge\max E_\kappa$, since $2qh\lambda\ge h^2/2$ and $2\kappa(h-\kappa)\le h^2/2$.
So the sets $E_\kappa$ with $\kappa=\lceil h/4\rceil+i\lambda\le 3h/4$, $i\ge0$, are pairwise disjoint; each has more than $\kappa(h-\kappa)\ge3h^2/16$ elements, and their number is
\[
\Bigl\lfloor\frac{3h/4-\lceil h/4\rceil}{\lambda}\Bigr\rfloor+1>\frac{2h-3}{4\lambda}\ge\frac{q(2h-3)}{h+4q}\ge\frac{q(2h-3)}{5h},
\]
using $\lceil h/4\rceil\le(h+3)/4$, $\lambda\le(h+4q)/(4q)$ and $q<h$.
Hence $|F_q|>\frac{3h^2}{16}\cdot\frac{q(2h-3)}{5h}=\frac{3qh(2h-3)}{80}\ge\frac{qh^2}{40}$ (since $h\ge 4$, which implies $3(2h-3)=6h-9\ge 2h$).

(iii) Fix $\alpha<\sqrt6$ and let $K-1=k=\lfloor\alpha\sqrt n\rfloor$.
For $q=0$, the empty subset sum gives $|F_0|\ge 1$. Combined with (ii) for $1\le q<h$, the left side of \eqref{eq:box} satisfies $\prod_{q<h}|F_q|\ge |F_0|\prod_{q=1}^{h-1}\frac{qh^2}{40}\ge(h-1)!\bigl(\frac{h^2}{40}\bigr)^{h-1}$, so its logarithm is at least $\log((h-1)!)+(h-1)\log(h^2/40)=3h\log h-O(h)=\frac32n\log n-O(n)$.
The bound for the right side in the proof of \Cref{thm:perm_threshold}(b) uses only $Q\le n$ and $r_q\le n$, so its logarithm is at most $\frac{\alpha^2}4n\log n+O(n)$.
Since $\alpha^2/4<3/2$, \eqref{eq:box} holds for large $n$, and \Cref{prop:box} gives $k^*(n)\ge\lfloor\alpha\sqrt n\rfloor+1$.
\end{proof}

\Cref{prop:subset}(iii) is the second part of \Cref{thm:perm_threshold}(b).
The proof of (i) uses only that $h$ is even, so it re-proves \Cref{lem:twin_squares}(1) for even $h$.
For $h$ even, $|F_q|$ is at least $\binom{h+1}2+1$ for every $q$, against $\lceil h/2\rceil+1$ in the count used for the first part.
The convergence is slow: with $|F_q|\ge\lceil r_q/2\rceil+1$, \eqref{eq:box} gives $k^*(n)\ge2$ at $n=14$ and $k^*(n)\ge7$ at $n=100$, and with the exact $|F_q|$ for the square of \Cref{prop:subset} it gives $k^*(40)\ge8$ and $k^*(100)\ge16$, against $k^*(40)=26$ and $k^*(100)\le45$ (\Cref{fig:kstar}).
Thus
\[
(\sqrt6-o(1))\sqrt n\le k^*(n)\le2\sqrt2\,\sqrt{n\log n}\qquad(4\mid n),
\]
with $\sqrt2$ in place of $\sqrt6$ for general $n$.

\subsection{Exact values}
\label{sec:kstar_exact}

\begin{proof}[Proof of \Cref{thm:perm_threshold}(c)]
\emph{$3\le n\le13$.}  Here $k^*(n)\le n-1$, and the lower bound needs a witness for $k^*(n)\ge n-1$, i.e.\ an arrangement satisfying the degree-$(n-2)$ conditions that is not semi-magic.  For $n=3$ and $n=4$,
\[
\begin{pmatrix}1&3&9\\8&6&5\\7&4&2\end{pmatrix}\ \ \gamma=(1,-2,1),\ \rho=(-2,4,-2);\qquad
\begin{pmatrix}14&2&5&13\\8&15&7&4\\12&3&9&10\\1&11&16&6\end{pmatrix}\ \ \gamma=(1,-3,3,-1),\ \rho=0.
\]
In both cases $P_\gamma=(1-t)^{n-1}$ and $P_\rho\in\{0,-2(1-t)^2\}$, so the degree-$(n-2)$ conditions hold.
For $5\le n\le13$ we found arrangements with $\rho=0$ and $\gamma_j=(-1)^j\binom{n-1}j$, i.e.\ $P_\gamma=(1-t)^{n-1}$ (Appendix~\ref{sec:app_witnesses}).

\emph{The lattices $\Lambda_{n,n-2}$ and $\Lambda_{n,n-3}$.}  $\Lambda_{n,n-2}$ has rank $1$ and is spanned by $v=\bigl((-1)^j\binom{n-1}j\bigr)_j$.  If $\gamma=av\ne0$ with $a\in\Z$, then $\max_j|\gamma_j|\ge\binom{n-1}{\lfloor(n-1)/2\rfloor}$, which exceeds $n^2(n-1)/2$ for $n\ge14$, contradicting \Cref{lem:col_feasible}.
(At $n=14$, $\binom{13}6=1716>1274$; passing from $n$ to $n+1$ multiplies the left side by $2$ for even $n$ and by $2n/(n+1)\ge15/8$ for odd $n\ge15$, and the right side by $(n+1)^2/(n(n-1))$, which decreases in $n$ and equals $225/182<15/8$ at $n=14$.)
So $k^*(n)\le n-2$ for $n\ge14$; at $n=13$ the bound fails ($\binom{12}6=924<1014$), and the witness has central entry $924$.

$\Lambda_{n,n-3}$ has rank $2$: $P_\gamma=(t-1)^{n-2}(a+bt)$ with $a,b\in\Z$.
At $t=-1$, $2^{n-2}|a-b|=|P_\gamma(-1)|\le\|\gamma\|_1\le N^2/4=n^4/4$ by \Cref{lem:col_feasible}; since $n^4<2^n$ for $n\ge17$, $a=b$ and $P_\gamma=a(t-1)^{n-2}(t+1)$.
By Parseval, $\|\gamma\|_2^2=a^2\sum_j\bigl(\binom{n-2}j-\binom{n-2}{j-1}\bigr)^2=2a^2\binom{2n-4}{n-2}/(n-1)$ (twice a Catalan number).
At $n=18$ this factor is $70{,}715{,}340>50{,}860{,}872=18^5(18^2-1)/12$.
From $n$ to $n+1$ the left side grows by the factor $2(2n-3)/n\ge11/3$ and the right side by $(n+1)^4(n+2)/(n^4(n-1))<3/2$.
So for $n\ge18$ the bound $\|\gamma\|_2^2\le n^5(n^2-1)/12$ of \Cref{lem:col_feasible} forces $a=0$, and $k^*(n)\le n-3$.
(At the next step $\Lambda_{n,n-4}$ has rank $3$, with $P_\gamma=(t-1)^{n-3}(a+bt+ct^2)$, where eliminating the three parameters $a,b,c$ requires more than evaluating at $t=-1$ and Parseval's identity, so from $n=18$ onward we turn to exact lattice enumeration.)

\emph{Lower bounds for $14\le n\le17$.}  For each of these $n$ we found an arrangement with all row sums $M(n)$ whose column deviations satisfy $(t-1)^{n-2}\mid P_\gamma$ and $\gamma\ne0$ (verified in exact integer arithmetic, Appendix~\ref{sec:appendix}; $\max_j|\gamma_j|$ is listed in \Cref{tab:kstar}).
\end{proof}

\begin{proof}[Proof of \Cref{thm:perm_threshold}(d)]
\emph{Upper bounds.}  If no nonzero vector of $\Lambda_{n,k}$ satisfies \eqref{eq:col_feasible} for every $p$, then $k^*(n)\le k$.
By \Cref{lem:col_feasible}, any column-deviation vector $\gamma$ sorted in decreasing order $\gamma_{(1)}\ge\dots\ge\gamma_{(n)}$ satisfies $\sum_{p=1}^P\gamma_{(p)}\le \frac{Pn(N-Pn)}{2}=\sum_{p=1}^P v_p$ for all $1\le P\le n$, with equality at $P=n$. Thus $\gamma$ is majorised by $v$, where $v_p=n(N-(2p-1)n)/2$. Because the squared Euclidean norm $x\mapsto\|x\|_2^2$ is Schur-convex, every column-feasible vector satisfies $\|\gamma\|_2^2\le\|v\|_2^2=n^5(n^2-1)/12=:R^2$. Only vectors within this Euclidean ball of radius $R$ need be examined; we enumerated them exactly with two independent implementations (Appendix~\ref{sec:appendix}).
At $k=k^*(n)$ the ball contains no nonzero vector for $16$ of the $23$ values of $n$, since the Gram--Schmidt norms of the reduced basis already exceed the radius.
For $n=20,22,24,26,34,36,38$ it contains $2,2,6,6,6,2,2$ vectors.
Each violates \eqref{eq:col_feasible}, in every case but one already with $p=1$; the exception is $n=26$, where two of the six vectors need $p=2$.
The ancillary certificate (Appendix~\ref{sec:appendix}) lists the reduced bases and, for each vector, a violated instance $(p,J)$.

\emph{Lower bounds.}  For each $18\le n\le40$ we found a witness for $k^*(n)\ge K$, with $K$ the value of \Cref{tab:kstar}: an arrangement with all row sums $M(n)$ whose column deviations satisfy $(t-1)^{K}\mid P_\gamma$ and $\gamma\ne0$; each was verified in exact integer arithmetic.
\end{proof}

\begin{table}[htbp]
\centering\footnotesize
\setlength{\tabcolsep}{2.9pt}
\begin{tabular}{r|rrrrrrrrrrrrrr}
$n$ & 14&15&16&17&18&19&20&21&22&23&24&25&26&27\\ \hline
$k^{*}(n)$ & 12&13&14&15&15&16&16&17&17&18&18&19&19&20\\
$\max_j|\gamma_j|$ & 297&572&1001&2002&1430&2860&1794&3094&2210&3638&2380&4522&2736&4541\\
$n^2(n-1)/2$ & 1274&1575&1920&2312&2754&3249&3800&4410&5082&5819&6624&7500&8450&9477\\[5pt]
$n$ & 28&29&30&31&32&33&34&35&36&37&38&39&40&\\ \hline
$k^{*}(n)$ & 21&21&22&22&23&23&23&24&24&25&25&26&26&\\
$\max_j|\gamma_j|$ & 8892&6405&11737&9330&14881&11799&9177&18147&14766&24150&16515&26520&22233&\\
$n^2(n-1)/2$ & 10584&11774&13050&14415&15872&17424&19074&20825&22680&24642&26714&28899&31200&
\end{tabular}
\caption{Exact values of the permutation threshold for $14\le n\le40$ (\Cref{thm:perm_threshold}(c),(d)), the largest column deviation $\max_j|\gamma_j|$ of the witness arrangement, and the bound $n^2(n-1)/2$ of \Cref{lem:col_feasible}. The upper bounds for $18\le n\le40$ were computed by two independent exact lattice enumerations and, for $n\le33$, also by integer programming (Appendix~\ref{sec:appendix}).}
\label{tab:kstar}
\end{table}

\begin{figure}[htbp]
\centering
\includegraphics[width=\textwidth]{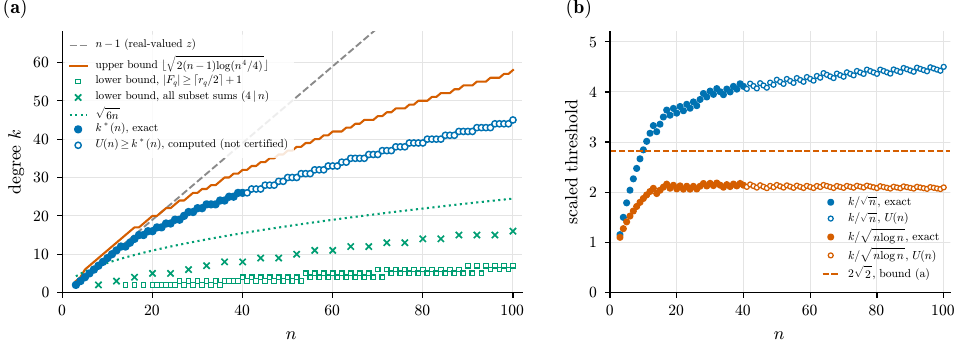}
\caption{The permutation threshold. (a) Exact values of $k^*(n)$ for $3\le n\le40$ (dots), computer-assisted upper bounds $U(n)$ for $41\le n\le100$ (two independent implementations agree, not certified; open circles), the linear-algebra threshold $n-1$, the bound of \Cref{thm:perm_threshold}(a), the lower bounds given by \eqref{eq:box} with $|F_q|\ge\lceil r_q/2\rceil+1$ (open squares, shown where they give $K\ge2$) and for $4\mid n$ with the square of \Cref{prop:subset} (crosses), and $\sqrt{6n}$, the asymptotic order of (b) for $4\mid n$. (b) The same exact values and upper bounds divided by $\sqrt n$ and by $\sqrt{n\log n}$.}
\label{fig:kstar}
\end{figure}

\paragraph{The witnesses.}
The target vectors $\gamma$ were chosen among the shortest vectors of $\Lambda_{n,k^*(n)-1}$ satisfying \eqref{eq:col_feasible}, as produced by the enumeration.
Starting from a semi-magic square of \Cref{lem:twin_squares} (or a Siamese (de la Loub\`ere) magic square \cite{andrews1960magic}), a steepest-descent search over exchanges within a row, which fix the row sums, minimised $\sum_j(\text{column sum}_j-M(n)-\gamma_j)^2$, with random exchanges to leave local minima; for the hardest cases a simulated annealing over arbitrary exchanges, with energy the squared distance of both row and column sums to their targets, was used.
Some witnesses come close to the bound of \Cref{lem:col_feasible}: $\max_j|\gamma_j|=924$ against $1014$ at $n=13$, and $24150$ against $24642$ at $n=37$.
For each $14\le n\le40$ the witness search succeeded at the last $k$ at which $\Lambda_{n,k}$ still contains a nonzero vector satisfying \eqref{eq:col_feasible}.
Let $U(n)$ be the least $k$ such that no nonzero vector of $\Lambda_{n,k}$ satisfies \eqref{eq:col_feasible} for every $J$; by \Cref{lem:col_feasible}, $k^*(n)\le U(n)$, and the proofs of (c) and (d) show $k^*(n)=U(n)$ for $3\le n\le40$.

\begin{conjecture}
\label{conj:lattice}
$k^*(n)=U(n)$ for every $n\ge3$.
\end{conjecture}

We tested \Cref{conj:lattice} for $41\le n\le46$, where both implementations give $U(n)=26,27,27,27,28,28$ (not certified).
For each of these $n$ the same search, run on the vectors of $\Lambda_{n,U(n)-1}$ in the ball $\|\gamma\|_2^2\le n^5(n^2-1)/12$ that satisfy \eqref{eq:col_feasible}, found a witness (verified in exact arithmetic; Appendix~\ref{sec:appendix}).
So $k^*(n)\ge U(n)$, i.e.\ $k^*(n)\ge26,27,27,27,28,28$, holds unconditionally for $41\le n\le46$; only the upper bounds $k^*(n)\le U(n)$ are uncertified, and the conjecture holds for these $n$ if they are correct.
The hardest case was $n=42$, where the ball contains a single such vector up to sign, with slack $7$ in one instance of \eqref{eq:col_feasible} with $p=3$; it and one of the $50$ vectors for $n=46$ were realised only by a longer two-stage annealing.
Beyond $n=40$ the values $U(n)$ are computer-assisted upper bounds for $k^*(n)$ (two independent implementations agree for every $n\le100$, but the values are labeled ``not certified'' in \Cref{fig:kstar} because the Fincke--Pohst tree visits up to millions of nodes, precluding an exportable lightweight certificate); for instance $k^*(50)\le30$, $k^*(80)\le39$ and $k^*(100)\le45$.
In \Cref{fig:kstar}(b), $k^*(n)/\sqrt n$ (and $U(n)/\sqrt n$ beyond $n=40$) grows from $3.2$ at $n=14$ to $4.5$ at $n=100$, while the same values divided by $\sqrt{n\log n}$ stay between $2.06$ and $2.18$ for $20\le n\le100$.
These data do not discriminate between the orders $\sqrt n$ and $\sqrt{n\log n}$: by \Cref{sec:kstar_upper}, the unconditional pigeonhole bound already gives $U(n)\ge(\sqrt{12}-o(1))\sqrt n\approx3.46\sqrt n$ (and Amoroso's bound gives $\sqrt{24}\approx4.90$), so a steady rise of $U(n)/\sqrt n$ across $n\le100$ is expected even if $U(n)$ is of order $\sqrt n$.

\section{The energy of a random arrangement}
\label{sec:random_energy}

Whereas Section~\ref{sec:theory} investigated how far up the axial moment hierarchy $\langle x^m,z\rangle=\langle y^m,z\rangle=0$ one must ascend before the permutation constraint forces all line sums to balance, the remainder of this paper examines the lowest-degree moments under a complementary, statistical perspective.
Specifically, Sections~\ref{sec:random_energy} and~\ref{sec:computations} focus on the first-degree axial moments ($m=1$) combined with the two main diagonals via the four-term energy $\Efour$, alongside Loly's second-degree ($m=2$) moment-of-inertia functional $\langle x^2+y^2,z\rangle$.
By analysing these low-degree observables over uniformly random arrangements, we obtain exact null moments, asymptotic distribution laws, and geometric estimates for the density of nearly balanced squares.

This section specialises classical results on permutation statistics.
The energies are quadratic forms $z^{\mathsf T}Az$ in a random permutation $z$ of fixed values, i.e.\ quadratic assignment statistics, a form that goes back to Daniels's generalised correlation coefficient \cite{daniels1944}; Bloemena \cite{bloemena1964} proved asymptotic normality for the number of edges of a graph joining randomly sampled points, and Mantel \cite{mantel1967} and Hubert \cite{hubert1987assignment} give the general mean and variance.
What is specific here is the evaluation of the invariants for the energies, with its dependence on the parity of $n$ (\Cref{thm:var_closed}).

Let $S$ be a uniformly random arrangement: $z_c=b_{\pi(c)}$ with $\pi$ uniform and $b$ the centred values $b_v=v-\frac{N+1}2$, $1\le v\le N$.
For a test function $\varphi$ put $T_\varphi=\langle\varphi,z\rangle=\sum_ca_cz_c$ with $a_c=\tilde\varphi_c/N$, a linear permutation statistic (\Cref{sec:bg_perm}); recall that $\tilde\varphi=\varphi-\bar\varphi\mathbf 1$ and $\sigma_n^2=(n^2+1)/12$ \eqref{eq:sigma}.
Throughout this section $\E$, $\Var$ and $\Cov$ refer to the random arrangement.

\subsection{Means}

\begin{proposition}[Means and covariance structure]
\label{prop:random_energy_moments}
For all test functions $\varphi,\psi$, $\E[T_\varphi]=0$ and $\E[T_\varphi T_\psi]=\sigma_n^2\langle\tilde\varphi,\tilde\psi\rangle$.
Consequently
\begin{gather}
\E\,\langle x,z\rangle^2=\E\,\langle y,z\rangle^2=\frac{n^4-1}{144},\qquad
\E\,\langle\mathbf 1_L,z\rangle^2=\frac{(n-1)(n^2+1)}{12\,n^2}\ \text{ for every line }L,
\label{eq:single_terms}\\
\E[\Efour]=\frac{(n-1)(n^2+1)(n^3+n^2+12)}{72\,n^2},\qquad
\E[\Efull]=\frac{n^4-1}{6\,n^2}.
\label{eq:mean_low}
\end{gather}
The covariance matrix of the vector of statistics in $\Efour$ (resp.\ $\Efull$) is $\sigma_n^2\Gamma$, with $\Gamma=(\langle\tilde\varphi_\ell,\tilde\varphi_{\ell'}\rangle)_{\ell,\ell'}$ the Gram matrix of the centred test functions $\varphi_\ell$ of the energy; its eigenvalues are
\begin{equation}
\Efour:\ \sigma_n^2\Bigl\{\tfrac{n^2-1}{12}\,(\times2),\ \tfrac{n-1}{n^2}\,(\times2)\Bigr\},
\qquad
\Efull:\ \sigma_n^2\Bigl\{\tfrac1n\,(\times\,2n-2),\ \tfrac{n-1}{n^2}\,(\times2),\ 0\,(\times2)\Bigr\}
\label{eq:weights}
\end{equation}
for odd $n$; for even $n$ the pair $\{\tfrac{n-1}{n^2},\tfrac{n-1}{n^2}\}$ is replaced by $\{\tfrac1n,\tfrac{n-2}{n^2}\}$.
\end{proposition}

\begin{proof}
$\sum_ca_c=0$ gives $\E T_\varphi=0$, and \eqref{eq:ww_identity} with $\sum_vb_v^2=N\|z\|^2$ gives $\E[T_\varphi T_\psi]=\frac1{N-1}\cdot\frac{\langle\tilde\varphi,\tilde\psi\rangle}{N}\cdot N\|z\|^2=\sigma_n^2\langle\tilde\varphi,\tilde\psi\rangle$.
The entries of $\Gamma$ follow from \Cref{lem:W_structure}: $\|\tilde{\mathbf 1}_L\|^2=(n-1)/n^2$, two lines meeting in $r$ cells have $\langle\tilde{\mathbf 1}_L,\tilde{\mathbf 1}_{L'}\rangle=r/N-1/n^2$, and $x$, $y$ are orthogonal to each other and to the centred diagonal indicators, since a diagonal meets every row and every column once.
The rows give the block $\frac1n(I-\frac1nJ)$, the columns likewise, and the means are the traces.
\end{proof}

The means are elementary (sampling without replacement): a line sum is the sum of a sample of size $n$ drawn without replacement from $\{1,\dots,N\}$, so $\E\delta_L^2=n\|z\|^2\frac{N-n}{N-1}$.
The two means differ only through the $x$, $y$ terms of $\Efour$ and the $2n$ row and column terms of $\Efull$, whose means are $(n^4-1)/72$ and $(n-1)(n^2+1)/(6n)$.
These agree if and only if $n(n+1)=12$, i.e.\ $n=3$, where both energies have mean $40/27$.

\begin{corollary}[Mean and variance of Loly's invariant]
\label{cor:loly_null}
For a uniformly random arrangement, $I(S)$ of \Cref{cor:loly} satisfies
\begin{equation}
\E[I(S)]=\frac{n^{2}(n^{4}-1)}{12},\qquad
\Var\bigl(I(S)\bigr)=N^{2}\sigma_n^{2}\,\bigl\|\widetilde{x^2+y^2}\bigr\|^2=\frac{n^{4}(n^{4}-1)(n^{2}-4)}{1080}.
\label{eq:loly_null}
\end{equation}
The mean equals Loly's value $n^2(n^4-1)/12$ for every random arrangement with $\E z_c=0$ for every cell (for instance any distribution invariant under complementation $s\mapsto N+1-s$); magic and semi-magic squares take this value exactly.
\end{corollary}

\begin{proof}
By \eqref{eq:loly_identity}, $I(S)-\frac{n^2(n^4-1)}{12}=N\,T_{x^2+y^2}$, so \Cref{prop:random_energy_moments} gives the mean and $\Var(I)=N^2\sigma_n^2\|\widetilde{x^2+y^2}\|^2$.
Since $x^2$ depends only on the column and $y^2$ only on the row, $\widetilde{x^2}\perp\widetilde{y^2}$, so $\|\widetilde{x^2+y^2}\|^2=2\bigl(\langle x^4,\mathbf 1\rangle-\langle x^2,\mathbf 1\rangle^2\bigr)$ with $\langle x^2,\mathbf 1\rangle=\frac{n^2-1}{12}$ and $\langle x^4,\mathbf 1\rangle=\frac{(n^2-1)(3n^2-7)}{240}$, which gives $\frac{(n^2-1)(n^2-4)}{90}$.
\end{proof}

For $n=3$, $\Var(I)=30$ and $I$ takes $29$ values from $46$ to $74$; for $n=4$, $\Var(I)=2176/3$ and $\Pr(I=340)=13049/450450\approx0.029$.

\subsection{Variances}
\label{sec:random_var}

Write a random energy as a quadratic form $E=\sum_\ell T_{\varphi_\ell}^2=z^{\mathsf T}Az$ with $A=\sum_\ell a_\ell a_\ell^{\mathsf T}$, where $a_\ell=\tilde\varphi_\ell/N$.
$A$ is symmetric with zero row sums.

\begin{proposition}[{Mantel \cite{mantel1967}, Hubert \cite{hubert1987assignment}; specialised to zero row sums and centred values}]
\label{prop:var_identity}
Let $A$ be a symmetric $N\times N$ matrix with zero row sums, $N\ge4$, and let $z$ be a uniformly random permutation of the centred values $b_v=v-\frac{N+1}2$, $1\le v\le N$.  Then $\E[z^{\mathsf T}Az]=\frac{N(N+1)}{12}\tr A$ and
\begin{equation}
\Var(z^{\mathsf T}Az)=-\frac{N(N+1)(2N+3)}{360}\,(\tr A)^2+\frac{N(N^2-1)(5N+6)}{360}\,\tr(A^2)-\frac{N^2(N+1)^2}{120}\sum_aA_{aa}^2 .
\label{eq:var_identity}
\end{equation}
\end{proposition}

\begin{proof}
Because $\sum_vb_v=0$, the power sums $p_r=\sum_vb_v^r$ satisfy $\E[z_a^2]=p_2/N$ and $\E[z_az_b]=-p_2/(N)_2$ for $a\ne b$, where $(N)_k=N(N-1)\cdots(N-k+1)$. Since each row of $A$ sums to zero, $\sum_{a\ne b}A_{ab}=-\tr A$, giving $\E[z^{\mathsf T}Az]=\frac{p_2}{N-1}\tr A=\frac{N(N+1)}{12}\tr A$.

To compute the second moment $\E[(z^{\mathsf T}Az)^2]=\sum_{a,b,c,d}A_{ab}A_{cd}\E[z_az_bz_cz_d]$, observe that the expectation $m_\nu=\E[z_az_bz_cz_d]$ depends only on the equality pattern of the $4$-tuple $(a,b,c,d)$, which is a set partition $\nu$ of $\{1,2,3,4\}$. Let $S^=(\nu)$ be the restricted sum of $A_{ab}A_{cd}$ over tuples with exact equality pattern $\nu$, and let $S^\ge(\nu')=\sum_{\nu\ge\nu'}S^=(\nu)$ be the \emph{unrestricted} sum over all tuples satisfying at least the equalities of $\nu'$. By M\"obius inversion on the partition lattice $\Pi_4$,
\[
\E[(z^{\mathsf T}Az)^2]=\sum_{\nu\in\Pi_4}m_\nu S^=(\nu)=\sum_{\nu'\in\Pi_4}\theta(\nu')S^\ge(\nu'),\qquad\text{where}\quad \theta(\nu')=\sum_{\nu\le\nu'}\mu(\nu,\nu')m_\nu.
\]
Here the zero-row-sum hypothesis produces a dramatic simplification: if a partition $\nu'$ contains any singleton block, the corresponding cell index appears in only one factor of $A_{ab}A_{cd}$ and is summed freely over all cells, yielding a row sum of $A$, which is zero! Thus $S^\ge(\nu')=0$ for all eleven partitions with singletons. Only the four singleton-free partitions survive: the three pair partitions $12|34$, $13|24$, $14|23$—which evaluate to $(\tr A)^2$, $\tr(A^2)$, and $\tr(A^2)$—and the single-block partition $1234$, which evaluates to $\sum_a A_{aa}^2$. Substituting the M\"obius coefficients $\theta(12|34)=m_{22}-2m_{211}+m_{1111}$ and $\theta(1234)=m_4-4m_{31}-3m_{22}+12m_{211}-6m_{1111}$ along with $p_2=N(N^2-1)/12$ and $p_4=N(N^2-1)(3N^2-7)/240$ (Appendix~\ref{sec:app_variance}) and subtracting $(\E[z^{\mathsf T}Az])^2$ yields \eqref{eq:var_identity}.
\end{proof}

For the energies $\Efour$ and $\Efull$, the three trace invariants are explicit rational functions of $n$ (Appendix~\ref{sec:app_variance}); even and odd orders $n$ give different rational expressions because the two main diagonals intersect at the centre cell if and only if $n$ is odd, yet their asymptotic expansions agree to the first three leading orders.

\begin{theorem}[Variances of the energies]
\label{thm:var_closed}
For every $n\ge3$,
\[
\Var(\Efour)=\begin{cases}\dfrac{(n-1)(n^2+1)\,V_{\mathrm o}(n)}{129600\,n^{6}}& n\text{ odd},\\[3mm]
\dfrac{(n^2+1)\,V_{\mathrm e}(n)}{129600\,n^{4}}& n\text{ even},\end{cases}\quad
\Var(\Efull)=\begin{cases}\dfrac{(n-1)(n^2+1)\,G_{\mathrm o}(n)}{180\,n^{6}}& n\text{ odd},\\[3mm]
\dfrac{(n-2)(n^2+1)\,G_{\mathrm e}(n)}{180\,n^{5}}& n\text{ even},\end{cases}
\]
where
\begin{align*}
V_{\mathrm o}&=25n^{11}+25n^{10}-82n^9-82n^8-9n^7-1209n^6+3738n^5-5862n^4\\&\qquad+2160n^3-2880n^2-4320n+4320,\\
V_{\mathrm e}&=25n^{10}-107n^8+73n^6-1200n^5+4947n^4-9600n^3+13782n^2-13680n+10080,\\
G_{\mathrm o}&=5n^6-5n^5-17n^4+11n^3-24n^2+18n+6,\qquad
G_{\mathrm e}=5n^5-12n^3+12n^2-23n+12 .
\end{align*}
For both parities,
\[
\Var(\Efour)=\frac{n^8}{5184}-\frac{41\,n^6}{64800}+O(n^4),\qquad
\Var(\Efull)=\frac{n^3}{36}-\frac{n^2}{18}-\frac{7n}{180}+O(1).
\]
\end{theorem}

\begin{proof}
Substituting $N=n^2$ and the explicit trace invariants \eqref{eq:inv_low}--\eqref{eq:inv_full} of Appendix~\ref{sec:app_variance} into \Cref{prop:var_identity} yields the rational variance formulas and their Laurent expansions.
\end{proof}

At $n=3$ the formulas give $\Var(\Efour)=1532/2187$ and $\Var(\Efull)=398/2187$, the values obtained by enumerating all $9!$ arrangements.

\subsection{Limit law}

\begin{corollary}[Limit law]
\label{cor:random_energy_clt}
Let $\varphi_1=x$, $\varphi_2=y$, $\varphi_3=\mathbf 1_{D_{\mathrm m}}$, $\varphi_4=\mathbf 1_{D_{\mathrm a}}$, $T_\ell=T_{\varphi_\ell}$, and let $\tau_\ell^2=\Var T_\ell$ (the diagonal of $\sigma_n^2\Gamma$).
As $n\to\infty$, $\bigl(T_1/\tau_1,\dots,T_4/\tau_4\bigr)$ converges in distribution to a standard Gaussian vector in $\R^4$, and
\begin{equation}
\frac{\Efour}{\E[\Efour]}\ \xrightarrow{\ d\ }\ \mathrm{Exp}(1).
\label{eq:low_limit}
\end{equation}
\end{corollary}

\begin{proof}
Fix $\alpha\in\R^4\setminus\{0\}$ and let $a=\sum_\ell\alpha_\ell a_\ell/\tau_\ell$ be the corresponding combination of the four vectors $a_\ell=\tilde\varphi_\ell/N$.
Then $\sum_\ell\alpha_\ell T_\ell/\tau_\ell=\sum_ca_cz_c=\sum_ca_cb_{\pi(c)}$ with $\sum_ca_c=\sum_vb_v=0$, and
\[
\biggl(N\,\frac{\max_ca_c^2\,\max_vb_v^2}{\sum_ca_c^2\,\sum_vb_v^2}\biggr)^{1/2}=\frac{\max_c|a_c|}{\|a\|_2}\cdot\frac{\max_v|b_v|}{(\sum_vb_v^2/N)^{1/2}}
\]
is the square root of the quotient in \eqref{eq:hoeffding_max}.
The second factor tends to $\sqrt3$.
For the first, $\max_c|a_c|\le\sum_\ell|\alpha_\ell|\max_c|a_{\ell,c}|/\tau_\ell$, and $\max_c|a_{\ell,c}|/\|a_\ell\|_2$ is $O(n^{-1})$ for $x$ and $y$ and $O(n^{-1/2})$ for a diagonal indicator (a centred indicator of $n$ of the $N$ cells).
Since $\Var T_\ell=\sigma_n^2N\|a_\ell\|_2^2$ by \eqref{eq:ww_identity}, $\|a_\ell\|_2/\tau_\ell=(\sigma_n\sqrt N)^{-1}$ for every $\ell$, so $\|a\|_2^2=\alpha^{\mathsf T}C_n\alpha/(N\sigma_n^2)$ with $C_n$ the correlation matrix of $(T_1,\dots,T_4)$, which tends to the identity matrix $I_4$ because the only nonzero off-diagonal correlation is $\operatorname{corr}(T_3,T_4)=-1/(n-1)$ for even $n$.
Hence $\max_c|a_c|/\|a\|_2\le\sum_\ell|\alpha_\ell|\bigl(\max_c|a_{\ell,c}|/\|a_\ell\|_2\bigr)/(\alpha^{\mathsf T}C_n\alpha)^{1/2}\to0$, condition \eqref{eq:hoeffding_max} holds, and \cite[Thm.~4]{hoeffding1951} gives asymptotic normality of the standardized linear combination $(\sum_\ell\alpha_\ell W_{n,\ell})/(\alpha^{\mathsf T}C_n\alpha)^{1/2}\Rightarrow\mathcal N(0,1)$ where $W_{n,\ell}=T_\ell/\tau_\ell$.
Because $\alpha^{\mathsf T}C_n\alpha\to\|\alpha\|_2^2$, Slutsky's theorem implies $\sum_\ell\alpha_\ell W_{n,\ell}\Rightarrow\mathcal N(0,\|\alpha\|_2^2)$, and the Cram\'er--Wold device establishes the joint Gaussian limit $(W_{n,1},\dots,W_{n,4})\Rightarrow\mathcal N(0,I_4)$.
In $\Efour=\tau_1^2(W_{n,1}^2+W_{n,2}^2)+\tau_3^2(W_{n,3}^2+W_{n,4}^2)$ (with $\tau_1=\tau_2$ and $\tau_3=\tau_4$), the diagonal weights satisfy $\tau_3^2/\tau_1^2=12/(n^2(n+1))\to0$; by the continuous mapping theorem, $\Efour/(2\tau_1^2)\to\frac12\chi^2_2=\mathrm{Exp}(1)$, and since $\E[\Efour]/(2\tau_1^2)\to1$, \eqref{eq:low_limit} follows. The continuous mapping theorem similarly establishes the joint limit distribution for the scaled energies $(N^3\Efour(Z_\pi), N^3\Efull(Z_\pi))$.
\end{proof}

The diagonal terms are thus negligible in the limit: their share $\tau_3^2/(\tau_1^2+\tau_3^2)=12/(n^3+n^2+12)$ of $\E[\Efour]$ is $0.25$, $0.074$ and $0.020$ at $n=3,5,8$, and the same argument shows that $\langle x,z\rangle^2+\langle y,z\rangle^2$, divided by its mean, has the same limit $\mathrm{Exp}(1)$.
The diagonals matter in the exact moments and in \Cref{sec:lowmode_n4}.

\begin{remark}[The $\chi^2$-mixture approximation at finite $n$]
\label{rem:gauss_mixture}
The $\chi^2$-mixture approximation, which treats the statistics of an energy as jointly Gaussian, replaces the energy by $\sum_jw_j\chi^2_{1,j}$, with independent $\chi^2_1$ variables and weights $w_j$ the eigenvalues \eqref{eq:weights}; its density can be computed by inversion of the characteristic function \cite{imhof1961} (\Cref{fig:random_energy}).
It has the exact mean but variance $2\sum_jw_j^2$, and by \Cref{thm:var_closed} the ratio of the exact variance to this value is
\[
1-\frac2n-\frac{7}{5n^2}+O(n^{-3})\quad(\Efull),\qquad 1-\frac{82}{25n^2}+O(n^{-3})\quad(\Efour).
\]
The $\chi^2$-mixture variance is $2\tr\bigl((\sigma_n^2\Gamma)^2\bigr)=N^2(N+1)^2\tr(A^2)/72$.
In \eqref{eq:var_identity} the coefficient of $\tr(A^2)$ is this value times $1-(4N+6)/(5N(N+1))=1-4/(5n^2)+O(n^{-4})$, while the $(\tr A)^2$ and $\sum_aA_{aa}^2$ terms are negative—reflecting the negative fourth-cumulant correlations induced by sampling the permutation entries without replacement.
For $\Efull$ these two negative terms produce the leading variance deficit $-4/(5n)-6/(5n)=-2/n$; for $\Efour$ all three terms contribute only at order $n^{-2}$ ($-4/5$, $-4/5$, $-42/25$, total $-82/25$).
For $\Efull$ the number of statistics grows with $n$ and \Cref{cor:random_energy_clt} does not apply.
The approximation is poor for small $n$: at $n=3$ the $\chi^2$-mixture standard deviation of $\Efull$ is about twice the exact one (\Cref{tab:random_energy}).
\end{remark}

\begin{table}[htbp]
\centering
\small
\begin{tabular}{c l r r r r}
\toprule
$n$ & energy & $\E[E]$ & $\sqrt{\Var E}$ & $\chi^2$-mixture sd & variance ratio \\
\midrule
3 & $\Efour$  & $40/27$ & 0.8370 & 1.171 & 0.511 \\
3 & $\Efull$ & $40/27$ & 0.4266 & 0.869 & 0.241 \\
4 & $\Efour$  & $391/96$ & 3.1284 & 3.586 & 0.761 \\
4 & $\Efull$ & $85/32$ & 0.9074 & 1.349 & 0.453 \\
5 & $\Efour$  & $234/25$ & 8.0483 & 8.694 & 0.857 \\
5 & $\Efull$ & $104/25$ & 1.3966 & 1.867 & 0.560 \\
6 & $\Efour$  & $2035/108$ & 17.122 & 18.007 & 0.904 \\
6 & $\Efull$ & $1295/216$ & 1.9671 & 2.459 & 0.640 \\
7 & $\Efour$  & $5050/147$ & 32.174 & 33.349 & 0.931 \\
7 & $\Efull$ & $400/49$ & 2.5697 & 3.089 & 0.692 \\
8 & $\Efour$  & $22295/384$ & 55.377 & 56.888 & 0.948 \\
8 & $\Efull$ & $1365/128$ & 3.2352 & 3.777 & 0.734 \\
\bottomrule
\end{tabular}
\caption{Exact means and standard deviations of the two energies (\Cref{prop:random_energy_moments}, \Cref{thm:var_closed}), the standard deviation of the $\chi^2$-mixture approximation of \Cref{rem:gauss_mixture}, and the ratio of the exact to the $\chi^2$-mixture variance.}
\label{tab:random_energy}
\end{table}

\begin{figure}[htbp]
\centering
\includegraphics[width=0.92\textwidth]{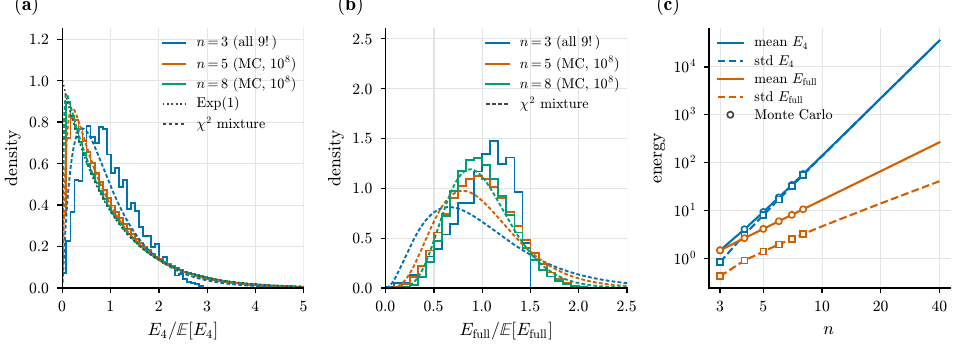}
\caption{Energy of a uniformly random arrangement. (a)~Distribution of $\Efour/\E[\Efour]$ for $n=3$ (all $9!$ arrangements) and $n=5,8$ ($10^8$ Monte Carlo samples each), with the densities of the $\chi^2$-mixture approximation (dashed) and the limit $\mathrm{Exp}(1)$ of \Cref{cor:random_energy_clt} (dotted). (b)~The same for $\Efull$, where the $\chi^2$-mixture approximation improves slowly (\Cref{rem:gauss_mixture}). (c)~Exact mean (solid) and standard deviation (dashed) of both energies for $3\le n\le40$, with Monte Carlo values for $3\le n\le8$ (open markers).}
\label{fig:random_energy}
\end{figure}

\section{Computations for small orders}
\label{sec:computations}

All counts below were computed in exact integer arithmetic; the methods are described in Appendix~\ref{sec:appendix}, and every number is produced by a script in the ancillary directory \texttt{anc/}.

\subsection{The classes \texorpdfstring{$L_3$}{L3} and \texorpdfstring{$L_4$}{L4}}
\label{sec:lowmode_n4}

Let $L_n$ be the class of arrangements whose first moments vanish and whose diagonal sums equal $M(n)$, i.e.\ the arrangements with $\Efour=0$: $\langle x,z\rangle=\langle y,z\rangle=0$ and both diagonal sums equal to $M(n)$.
With cells indexed by $(i,j)$, these are the conditions $U_1=U_2=\frac{n-1}2\cdot\frac{N(N+1)}2$ and $U_3=U_4=M(n)$ on the sums
$U_1=\sum_{i,j}j\,s_{ij}$, $U_2=\sum_{i,j}i\,s_{ij}$, $U_3=\sum_is_{ii}$, $U_4=\sum_is_{i,n-1-i}$, which are additive over cells.

\begin{proposition}
\label{prop:n3_lowmode}
For $n=3$, $\Efour(S)=0$ if and only if $S$ is magic; that is, $L_3$ consists of the eight magic squares.
\end{proposition}

\begin{proof}
Write the centred entries $z_c=s_c-5\in\{-4,\dots,4\}$ as $\left(\begin{smallmatrix}a&b&c\\d&e&f\\g&h&i\end{smallmatrix}\right)$, so that $\sum_cz_c=0$, $\sum_cz_c^2=60$, and every line sum of $z$ is the deviation of that line.
The degree-$1$ conditions with $\sum_j\gamma_j=\sum_i\rho_i=0$ give $\gamma=(\xi,-2\xi,\xi)$ and $\rho=(\eta,-2\eta,\eta)$, and the diagonal conditions give $a+e+i=c+e+g=0$.
The middle row, the middle column and the two diagonals cover every cell once and the centre three more times, so $3e=-2(\xi+\eta)$; as $e$ is an integer, $\xi+\eta=3\kappa$ and $e=-2\kappa$ with $\kappa\in\Z$.
Put $a'=a-\kappa$, $c'=c-\kappa$, $\xi'=\xi-\kappa$, $\eta'=\eta-\kappa$, so that $\xi'+\eta'=\kappa$.
The diagonals give $i=\kappa-a'$, $g=\kappa-c'$; the first and last rows (sum $\eta$) and columns (sum $\xi$) then give
\[
z=\begin{pmatrix}\kappa+a'&-\xi'-a'-c'&\kappa+c'\\-\eta'-a'+c'&-2\kappa&-\eta'+a'-c'\\\kappa-c'&-\xi'+a'+c'&\kappa-a'\end{pmatrix},
\qquad
\sum_cz_c^2=2\bigl(4\kappa^2+3(a'^2+c'^2)+\xi'^2+\eta'^2\bigr),
\]
so $4\kappa^2+3(a'^2+c'^2)+\xi'^2+\eta'^2=30$.
Because the four corners of $z$ equal $\kappa\pm a'$ and $\kappa\pm c'$, pairwise distinctness of the nine arrangement entries forces $a',c'\in\Z\setminus\{0\}$ with $|a'|\ne|c'|$ (since $a'=0$, $c'=0$, or $|a'|=|c'|$ would equate at least two corners), whence $a'^2+c'^2\ge 1^2+2^2=5$. Therefore $4\kappa^2+\xi'^2+\eta'^2$ is a multiple of $3$ and at most $15$, and $a'^2+c'^2=10-(4\kappa^2+\xi'^2+\eta'^2)/3$.
So $|\kappa|\le1$.
If $|\kappa|=1$, then $\xi'+\eta'=\pm1$ makes $\xi'^2+\eta'^2$ odd, so $4\kappa^2+\xi'^2+\eta'^2\in\{9,15\}$ (as it is at least $5$) and $\xi'^2+\eta'^2\in\{5,11\}$; $11$ is not a sum of two squares, and $\xi'^2+\eta'^2=5$ gives $a'^2+c'^2=7$, which is not either.
So $\kappa=0$, $\eta'=-\xi'$, and $2\xi'^2\le15$ divisible by $3$ forces $\xi'=0$.
Hence $\xi=\eta=0$: all rows, columns and diagonals have deviation $0$, and $S$ is magic.
The converse is clear.
\end{proof}

\begin{remark}
\label{rem:n3_enum}
Although the four conditions $\Efour(S)=0$ are strictly weaker than the $2n$ magic conditions for real arrays $z$, the proof above uses that the entries are distinct integers satisfying $\sum_cz_c^2=60$. Among all $9!=362{,}880$ arrangements of order $3$, exactly $760$ satisfy the degree-$1$ conditions, $4{,}224$ have both diagonal sums equal to $15$, $72$ (the semi-magic squares) satisfy the degree-$2$ conditions, and their intersection with the diagonal condition consists solely of the $8$ magic squares.
\end{remark}

For $n=4$ the equivalence fails (\Cref{rem:four_term_insufficient}).

\begin{proposition}[The size of $L_4$]
\label{prop:L4}
$|L_3|=8$ (\Cref{prop:n3_lowmode}), and
\[
|L_4|=14{,}639{,}440=8\times1{,}829{,}930,
\]
of which $7{,}040$ are magic squares.
The class $\{\langle x,z\rangle=\langle y,z\rangle=0\}$ (the arrangements whose lifted point sets in \cite{mathewson_gems} have diagonal second-moment matrix) has $6{,}697{,}177{,}424$ members for $n=4$, and the class with both diagonal sums equal to $34$ has $49{,}932{,}288{,}000$.
\end{proposition}

\begin{proof}[Computation]
A meet-in-the-middle enumeration over $(U_1,\dots,U_4)$, splitting the grid into the top two and bottom two rows (Appendix~\ref{sec:appendix}); the matching pairs were expanded to obtain the breakdown of \Cref{tab:L4}.
The same program reproduces the brute-force counts for $n=3$ ($8$, $760$, $4224$, and $72$ semi-magic squares) and, with the $2n$ line sums in place of $U_1,U_2$, the $7{,}040$ magic and $549{,}504$ semi-magic squares of order $4$.
Two further programs written independently (Appendix~\ref{sec:appendix}) together reproduce $|L_4|$, the first-moment count $6{,}697{,}177{,}424$ and every entry of \Cref{tab:L4}.
The diagonal-only count also follows from a direct count: there are $2150$ ordered pairs of disjoint $4$-subsets of $\{1,\dots,16\}$ with sum $34$, and $2150\cdot(4!)^2\cdot8!=49{,}932{,}288{,}000$.
\end{proof}

Since $D_4$ acts freely on the set of all arrangements (entries are distinct) and preserves $L_4$, $L_4$ consists of $1{,}829{,}930$ orbits of size $8$.

\begin{table}[htbp]
\centering
\small
\begin{tabular}{lrrrr}
\toprule
 & \multicolumn{4}{c}{columns with sum 34} \\
\cmidrule(lr){2-5}
rows with sum 34 & 0 & 1 & 4 & total \\
\midrule
0 & 9\,637\,312 & 2\,172\,824 & 100\,728 & 11\,910\,864 \\
1 & 2\,172\,824 & 429\,312 & 9\,336 & 2\,611\,472 \\
4 & 100\,728 & 9\,336 & \textbf{7\,040} (magic) & 117\,104 \\
\midrule
total & 11\,910\,864 & 2\,611\,472 & 117\,104 & 14\,639\,440 \\
\bottomrule
\end{tabular}
\caption{The members of $L_4$ by the number of rows and of columns with sum $M(4)=34$; the table is symmetric under transposition.}
\label{tab:L4}
\end{table}

Exactly two or three rows with sum $34$ are impossible in $L_4$: the row deviations satisfy $\sum_i\rho_i=0$ and $\sum_ii\rho_i=0$, which cannot hold with exactly one or two nonzero $\rho_i$; the same holds for columns.
Thus $7040/14{,}639{,}440\approx0.048\%$ of $L_4$ is magic, and $9{,}637{,}312$ members ($66\%$) have no row or column with sum $34$, among them the arrangement of \Cref{rem:four_term_insufficient}.

\paragraph{A local-limit heuristic.}
Under a uniformly random arrangement the vector $(U_1,\dots,U_4)$ has mean vector $\mu_U$ equal to the target vector and covariance matrix $\Sigma_U$ given by \eqref{eq:ww_identity}: $U_1,U_2$ are uncorrelated with each other and with the diagonal sums, which have correlation $-1/(n-1)$ for even $n$ and $0$ for odd $n$.
The target mean $\mu_U$ belongs to the integer lattice $\Z^4$ for all $n$: $U_3=U_4=M(n)=n(n^2+1)/2$ is an integer for every $n$, while $U_1=U_2=\frac{n^2(n-1)(n^2+1)}{4}$ is also an integer for every $n$ (if $n$ is even, $4\mid n^2$; if $n$ is odd, both $n-1$ and $n^2+1$ are even, so $4\mid (n-1)(n^2+1)$).
Furthermore, transpositions of consecutive values $\{v, v+1\}$ generate the full lattice $\Z^4$: choosing adjacent cells in row $i$ off both diagonals yields increment vector $(1,0,0,0)=e_1$; adjacent cells in column $j$ off both diagonals yields $e_2$; and swapping a cell on $D_{\mathrm m}$ (resp.\ $D_{\mathrm a}$) with an off-diagonal cell yields $e_3$ (resp.\ $e_4$). Thus the weight-difference lattice has covolume $1$.
Because $\Sigma_U$ is asymptotically block-diagonal, its determinant scales as
\begin{equation}
\det(\Sigma_U)^{1/2}\sim\Var(U_1)\Var(U_3)\sim\Bigl(\frac{n^8}{144}\Bigr)\Bigl(\frac{n^5}{12}\Bigr)=\frac{n^{13}}{1728},
\label{eq:det_sigma_scaling}
\end{equation}
so a Gaussian local limit gives the heuristic estimate $|L_n|\approx N!\,(2\pi)^{-2}\det(\Sigma_U)^{-1/2}\asymp N!\,n^{-13}$.
For $n=3$ the estimate is $13.6$, yielding the explicit ratio $8/13.6\approx0.59$ (against $|L_4|/(1.82\times10^7)\approx0.80$ at $n=4$); for the two-condition classes the ratios at $n=3,4$ are $0.59,0.91$ (first moments) and $1.10,0.96$ (diagonals).
For $n=5,6$ Monte Carlo ($10^{10}$ and $2\times10^{10}$ samples, $328$ and $62$ members of $L_n$) gives $|L_5|\approx5.1\times10^{17}$ ($\pm6\%$) and $|L_6|\approx1.2\times10^{33}$ ($\pm13\%$), ratios $0.76$ and $0.78$ to the estimate.
Local limit theorems for permutation statistics of this kind are studied in \cite{kuperberg2017probabilistic,eberhard2019additive}, and related counts of matrices with prescribed margins in \cite{canfield2010asymptotic}.

\subsection{Single swaps}
\label{sec:swap-gaps}

\begin{remark}[Single-swap formula]
\label{rem:swap}
Let $S^{ab}$ be obtained from $S$ by exchanging the entries of cells $a\ne b$, let $d=s_b-s_a$, let $\lambda(a,b)$ be the number of lines containing exactly one of $a,b$, and $u_S(c)=\sum_{L\ni c}\delta_L(S)$.
Expanding \eqref{eq:Efull_devs},
\[
N^{2}\bigl(\Efull(S^{ab})-\Efull(S)\bigr)=d^{2}\lambda(a,b)+2d\bigl(u_S(a)-u_S(b)\bigr),
\]
the incremental evaluation used in local search \cite{pinn1998number}.
Here $\lambda(a,b)=2[i\ne i']+2[j\ne j']+[\text{exactly one of }a,b\text{ lies on }D_{\mathrm m}]+[\text{exactly one of }a,b\text{ lies on }D_{\mathrm a}]\in\{2,\dots,6\}$ for $a=(i,j)$, $b=(i',j')$.
For a magic square $S$, $N^2\Efull(S^{ab})=d^2\lambda(a,b)\ge2$, so no two magic squares differ by a transposition, and the increase of $N^2\Efull$ under a transposition lies in the set $\{d^2\lambda:1\le d\le N-1,\ 2\le\lambda\le6\}$.
The smallest increase $g_{\min}(S)=\min_{a\ne b}d^2\lambda(a,b)$ satisfies $2\le g_{\min}(S)\le6$: the value $2$ occurs if and only if two consecutive values lie in a common row or column and off both diagonals, and consecutive values give at most $6$.
\end{remark}

Call $S$ a local minimum with respect to transpositions if no transposition lowers the energy.
For $n=3$ the local minima of $\Efull$ are the eight magic squares and eight non-magic arrangements with $N^2\Efull=2$, all strict; $\Efour$ has $344$ local minima, of which $336$ are non-magic, and only $72$ of them ($64$ non-magic) are strict.

\section{Open questions}
\label{sec:discussion}

The main results are the lower bounds $k^*(n)\ge(\sqrt2-o(1))\sqrt n$, and $k^*(n)\ge(\sqrt6-o(1))\sqrt n$ for $4\mid n$, for the permutation threshold, which together with Schur's bound give $(\sqrt2-o(1))\sqrt n\le k^*(n)\le2\sqrt2\,\sqrt{n\log n}$; the exact values for $n\le40$ (computer-assisted when $18\le n\le40$); and exact second moments of the line-sum energies of a random arrangement.
The following questions remain.

\begin{enumerate}\setlength\itemsep{2pt}
\item \emph{The order of $k^*(n)$.} Is $k^*(n)$ of order $\sqrt n$ or of order $\sqrt{n\log n}$?
The question has two parts.
The upper bound uses only that $\gamma$ is a nonzero integer vector with $\|\gamma\|_1\le n^4/4$, and among the vectors satisfying \eqref{eq:col_feasible} there are, for large $n$, vectors vanishing to order $(\sqrt{24}-o(1))\sqrt n$ at $1$ (\Cref{sec:kstar_upper}); whether there are vectors of larger order, up to $\sqrt{n\log n}$, is the known logarithmic gap between the bounds for the least height of an integer polynomial with a zero of high order at $1$ \cite{mignotte1982,bombieri1987low,amoroso1995}.
The part specific to arrangements is realisability: which vectors satisfying \eqref{eq:col_feasible} are column-deviation vectors.
Our constructions realise vectors with a zero of order $(\sqrt2-o(1))\sqrt n$ at $1$, and $(\sqrt6-o(1))\sqrt n$ for $4\mid n$; \Cref{conj:lattice} asserts that realisability is no restriction at the threshold, so that $k^*(n)=U(n)$ and the question reduces to the lattice problem.
\item \emph{The lattice threshold.} Prove \Cref{conj:lattice}, or find $n$ with $k^*(n)<U(n)$.  It holds for $n\le40$; for $41\le n\le46$ witnesses prove $k^*(n)\ge U(n)$, and equality holds if the uncertified values of $U(n)$ are correct.
\item \emph{The size of $L_n$.} Is $|L_n|\sim N!\,(2\pi)^{-2}\det(\Sigma_U)^{-1/2}\,c_n$ with $c_n$ bounded away from $0$ and $\infty$, i.e.\ does a local limit theorem hold for the four statistics $(U_1,\dots,U_4)$ of \Cref{sec:lowmode_n4}?  The ratio $c_n$ is $0.59$, $0.80$ for $n=3,4$ and about $0.76$, $0.78$ for $n=5,6$.
\item \emph{Other families.} The orthogonality formulation applies to pandiagonal and Franklin squares \cite{loly2007franklin}, with $W$ replaced by the span of the relevant line indicators.  The corresponding thresholds and random energies have not been worked out.
\end{enumerate}

\paragraph{Funding.} This research did not receive any specific grant from funding agencies in the public, commercial, or not-for-profit sectors.

\paragraph{Declaration of competing interest.} The author declares that he has no known competing financial interests or personal relationships that could have appeared to influence the work reported in this paper.

\paragraph{Data availability.} All verification scripts, witness arrangements, lattice enumeration certificates, and source code supporting the findings of this study are included in the ancillary directory and are openly available at \url{https://github.com/kylemath/magicGem}.

\section*{Declaration of generative AI and AI-assisted technologies in the manuscript preparation process}
\label{sec:ai_statement}
During the preparation of this work the author used Anthropic's Claude and Google's Gemini to assist with mathematical exploration, software implementation, and manuscript editing. After using these tools, the author reviewed and edited the content as needed, verified all proofs and exact computational certificates, and takes full responsibility for the content of the published article.

\bibliographystyle{plain}
\bibliography{references}

\appendix
\section{Computational methods}
\label{sec:appendix}

All computations were run on an Apple M4 Max under macOS with Python~3.9 (NumPy, SciPy, SymPy) and C compiled with Apple clang; uniformly random arrangements were generated by Fisher--Yates shuffles.
The scripts and data are in \texttt{anc/}, with a \texttt{README} listing the commands that regenerate every table and figure.

\paragraph{The permutation threshold.}
\texttt{anc/kstar\_verify.py} checks in exact integer arithmetic that each witness arrangement ($5\le n\le40$) is a permutation of $1,\dots,N$ with all row sums $M(n)$, and computes the order of vanishing of $P_\gamma$ at $1$ from the binomial moments $\sum_j\binom jr\gamma_j$.
The standard set $B_{\mathrm{std}} = \{t^i(t-1)^{k+1}\}_{0 \le i < n-k-1}$ is a $\Z$-basis of $\Lambda_{n,k}$: because $(t-1)^{k+1}$ is monic in $\Z[t]$, polynomial division of any $P_\gamma(t) \in \Z[t]_{<n}$ divisible by $(t-1)^{k+1}$ in $\R[t]$ yields an integer quotient $Q(t) \in \Z[t]_{<n-k-1}$, and the $(n-k-1)\times n$ coefficient matrix of $B_{\mathrm{std}}$ is upper-triangular with diagonal entries $(-1)^{k+1} = \pm 1$.
For the upper bounds, the basis $\{t^i(t-1)^{k+1}\}$ of $\Lambda_{n,k}$ is reduced by LLL (exact integer arithmetic) and all lattice vectors with $\|\gamma\|_2^2\le n^5(n^2-1)/12$ are enumerated by the Fincke--Pohst method with exact rational Gram--Schmidt coefficients, testing \eqref{eq:col_feasible} for every $p$.
Since $\Lambda_{n,k+1}\subset\Lambda_{n,k}$, the existence of a nonzero vector satisfying \eqref{eq:col_feasible} is monotone in $k$, and the least $k$ for which $\Lambda_{n,k}$ contains no such vector is the upper bound.
A second implementation, written independently (\texttt{anc/kstar\_enum\_independent.py}, with its own exact LLL reduction and enumeration), reproduces the upper bounds for $14\le n\le40$ and writes the certificate \texttt{anc/kstar\_certificates.json}: for each $n$ the reduced basis of $\Lambda_{n,k^*(n)}$, the exact squared Gram--Schmidt norms, and every nonzero vector of the ball with a violated instance $(p,J)$ of \eqref{eq:col_feasible}.
The short script \texttt{anc/kstar\_check\_certificates.py} checks in exact arithmetic that each basis $B = \{b_0,\dots,b_{r-1}\}$ lies in $\Lambda_{n,k^*(n)}$ and has the same Gram determinant as $B_{\mathrm{std}}$, $\prod_{i=0}^{r-1} \|b_i^*\|_2^2 = \det(B_{\mathrm{std}} B_{\mathrm{std}}^{\mathsf T})$, which certifies that the index $[\Lambda_{n,k^*(n)} : \mathrm{span}_\Z(B)] = 1$ and hence that $B$ is a $\Z$-basis.
With Gram--Schmidt orthogonal vectors $b_j = b_j^* + \sum_{i<j} \mu_{ji} b_i^*$, every nonzero lattice vector $v = \sum_{j=0}^{r-1} x_j b_j \in \Lambda_{n,k^*(n)} \setminus \{0\}$ expands as
\[
\|v\|_2^2 = \sum_{i=0}^{r-1} \Bigl(x_i + \sum_{j>i} \mu_{ji} x_j\Bigr)^2 \|b_i^*\|_2^2 \ge x_{i_{\max}}^2 \|b_{i_{\max}}^*\|_2^2 \ge \min_{0\le i < r} \|b_i^*\|_2^2 ,
\]
where $i_{\max} = \max\{i : x_i \ne 0\}$ (so the inner sum vanishes at $i_{\max}$ and $x_{i_{\max}}^2 \ge 1$).
When $\min_i \|b_i^*\|_2^2 > n^5(n^2-1)/12 =: R^2$, the ball contains no nonzero lattice vector.
Otherwise, the Fincke--Pohst algorithm recursively enumerates integers $x_i$ with $(x_i - c_i)^2 \|b_i^*\|_2^2 \le \mathrm{rem}_i$ around the rational centers $c_i = -\sum_{j>i} \mu_{ji} x_j$.
Because the input vectors $b_i$ are integer vectors, all $b_i^*$, $\|b_i^*\|_2^2$, $\mu_{ji}$, and interval endpoints lie in $\mathbb{Q}$, so Fincke--Pohst operates purely in exact rational arithmetic with zero rounding error.
The script checks that the enumerated vectors are precisely the listed vectors (this happens for $n\in\{20,22,24,26,34,36,38\}$), and checks each violated instance.
For $n\le33$ the upper bounds were also reproduced by floating-point integer programming (\texttt{anc/kstar\_crosscheck/}).
The witnesses were found with \texttt{anc/realise.py} and \texttt{anc/anneal2.c}, starting from the squares of \Cref{lem:twin_squares} (\texttt{anc/constructions.py}; an independent implementation of the construction is \texttt{anc/twin\_squares.py}).
\texttt{anc/kstar\_subset\_sums.py} checks \Cref{prop:subset}(i) for $h\le100$ and the counting in the proof of (ii) for $h\le60$, and evaluates \eqref{eq:box} with the exact $|F_q|$; \texttt{anc/kstar\_beyond40.py} reruns or checks the tests of \Cref{conj:lattice} for $41\le n\le46$, which were run under Linux.

\paragraph{Variances.}
\texttt{anc/variance\_proof.py} derives \eqref{eq:var_identity} symbolically by M\"obius inversion over the $15$ set partitions of $\{1,2,3,4\}$, checks it by enumerating all permutations for random matrices with $N\le7$, computes the invariants \eqref{eq:inv_low}--\eqref{eq:inv_full} symbolically and checks them against the explicit matrices for $3\le n\le25$, and compares the result with the enumeration of all $9!$ arrangements for $n=3$.
Monte Carlo with $10^8$ arrangements per order ($3\le n\le8$) agrees with the exact means and standard deviations to within $5\times10^{-4}$ relative error (\texttt{anc/mc\_results.json}).
The densities in \Cref{fig:random_energy} were computed by numerical inversion of the characteristic function \cite{imhof1961}.

\paragraph{Meet-in-the-middle enumeration of $L_4$.}
The $16$ cells are split into the top two rows $A$ and the bottom two rows $B$.
For each of the $\binom{16}{8}=12{,}870$ value sets $V$ and each of the $8!$ bijections $V\to A$ the partial vector $(U_1,\dots,U_4)|_A$ is packed into an integer key; for the complementary set and each bijection $V^c\to B$ the key of the target minus the partial vector is formed; the two key lists are sorted and merged, and the sum over matching keys of the products of multiplicities is $|L_4|$.
This touches about $1.04\times10^9$ partial vectors instead of $16!\approx2.1\times10^{13}$ arrangements (\texttt{anc/lowmode\_enum.c}).
The NumPy programs \texttt{anc/lowmode\_recount.py} and \texttt{anc/table4\_recount.py} use the same split with a different packing of the keys and tag each half by its number of rows (and the sum of a chosen column) equal to $34$; together with a direct enumeration of the members of $L_4$ whose four rows all sum to $34$, they recount $|L_4|$, the first-moment class and every entry of \Cref{tab:L4}.

\texttt{anc/verify\_numbers.py}, \texttt{anc/n3\_lowmode\_proof.py} and \texttt{anc/algorithms.py} recompute the remaining small-order counts and local-minimum statistics in exact arithmetic.

\section{Squares with many complementary rows}
\label{sec:app_twin}

\begin{proof}[Proof of \Cref{lem:twin_squares}(1)]
Give the value $v$ the weight $|2v-(N+1)|$, an odd number in $\{1,3,\dots,N-1\}$; each weight $w$ belongs to exactly the two complementary values $(N+1\pm w)/2$.
Index the weights by $\iota=(w-1)/2\in\{0,\dots,2h^2-1\}$, $w_\iota=2\iota+1$.
It suffices to partition the indices into $h$ sets $R_0,\dots,R_{h-1}$ of size $n$ and choose signs $\sigma_\iota\in\{\pm1\}$ with $\sum_{\iota\in R_r}\sigma_\iota w_\iota=0$ for every $r$: then row $r$ receives the values $(N+1+\sigma_\iota w_\iota)/2$, $\iota\in R_r$, in any order, and row $h+r$ their complements in the same columns.
Every value occurs once, row $r$ sums to $n\frac{N+1}2=M(n)$, row $h+r$ to $n(N+1)-M(n)=M(n)$, and every column to $h(N+1)=M(n)$.
A \emph{domino} $\{\iota,\iota+1\}$ (for odd $h$, any two consecutive unused indices; they need not be of the form $\{2t,2t+1\}$) with signs $(+,-)$ or $(-,+)$ contributes $\mp2$.
If $h$ is even, give each row $h$ of the dominoes $\{2t,2t+1\}$ with alternating sign patterns.
If $h\ge7$ is odd, $h-3$ is even, so assigning $h-3$ dominoes with alternating sign patterns $(+,-)$ and $(-,+)$ contributes $\frac{h-3}{2}(-2+2)=0$ to the row weight sum.
Give row $r$ the four indices $4r,\dots,4r+3$ with signs $(+,+,+,-)$, contributing $(8r+1)+(8r+3)+(8r+5)-(8r+7)=16r+2$; the two indices $\alpha_r=5h-1-r$ and $\beta_r=5h+7r$ with signs $(+,-)$, contributing $2(\alpha_r-\beta_r)=-16r-2$; and $h-3$ dominoes with alternating sign patterns, giving total signed weight sum $(16r+2)+(-16r-2)+0=0$.
The quadruples fill $[0,4h-1]$ and the $\alpha_r$ fill $[4h,5h-1]$.
In the range $[5h, 2h^2-1]$, the only used indices are $\beta_r = 5h+7r$ ($0\le r < h$); consecutive $\beta_r$ and $\beta_{r+1}$ leave gaps of exactly $6$ consecutive unused indices, each partitioning into $3$ consecutive dominoes $\{\iota,\iota+1\}$, while the remaining interval $[12h-6,\,2h^2-1]$ consists of consecutive indices of even length $2h^2-12h+6 = 2(h^2-6h+3) \ge 0$ (since $h\ge7$), which partitions into $h^2-6h+3$ consecutive dominoes.
So all unused indices in $[5h, 2h^2-1]$ partition into $3(h-1)+(h^2-6h+3)=h(h-3)$ consecutive dominoes, as required.
\end{proof}

\begin{proof}[Proof of \Cref{lem:twin_squares}(2)]
Start with a $7\times7$ magic square and border it successively to the orders $9,11,\dots,n$ by the following step, which turns a semi-magic square $T$ of order $n-2$ into one of order $n=2h+1$ ($h\ge4$) whose top and bottom rows are complementary on the inner columns.
Let $N=n^2$, $c=(N+1)/2$, $A=4h$, $B=4h-2$.
The inner block is $T+(2n-2)$, which uses the values $2n-1,\dots,N-2n+2$; each inner row and column then sums to $M(n-2)+(n-2)(2n-2)=(n-2)(N+1)/2$.
The border uses the $4h=2n-2$ complementary pairs $\{p,N+1-p\}$, $1\le p\le4h$.
Put $s_{00}=N+1-A$, $s_{n-1,n-1}=A$, $s_{0,n-1}=B$, $s_{n-1,0}=N+1-B$.
Let $\mathcal P=[2h,4h]\setminus\{A,B\}$ and $\mathcal Q=[1,2h-1]$, both of size $2h-1$, so $\sum\mathcal P-\sum\mathcal Q=4h^2-4h+2$.
Choose an $h$-subset $T_+$ of $[2h,4h-3]\subset\mathcal P$ and an $(h-1)$-subset $T_-$ of $\mathcal Q$ with $\sum T_+-\sum T_-=c-A+B=2h^2+2h-1$.
This is possible: the subsets of a given size of an interval of integers have as sums all integers between the smallest and the largest such sum (raise one element by $1$ at a time), so $\sum T_+-\sum T_-$ takes every integer in $[h^2+h,\,3h^2-2h]$, which contains $2h^2+2h-1$ for $h\ge4$.
The inner top row consists of the values $N+1-p$ ($p\in T_+$) and $p$ ($p\in T_-$); the inner bottom row of their complements in the same columns; the inner left column of $N+1-p$ ($p\in\mathcal P\setminus T_+$) and $p$ ($p\in\mathcal Q\setminus T_-$); and the inner right column of their complements in the same rows.
The top row sums to $(h+1)(N+1)-A+B-\sum T_++\sum T_-=(2h+1)c=M(n)$, and the left column to $(h+1)(N+1)-A-B-\sum(\mathcal P\setminus T_+)+\sum(\mathcal Q\setminus T_-)=M(n)$, because $\sum(\mathcal P\setminus T_+)-\sum(\mathcal Q\setminus T_-)=4h^2-4h+2-(2h^2+2h-1)=c-A-B$.
The bottom row and right column then sum to $n(N+1)-M(n)=M(n)$, each inner row and column to $N+1+(n-2)(N+1)/2=M(n)$, and every value $1,\dots,N$ occurs once.
Bordering shifts the inner square by a constant $\omega$, and $(u+\omega)+(u'+\omega)=N+1$ whenever $u+u'$ equals the previous $N+1$, so earlier complementary relations persist; this gives the complementary rows $\ell,n-1-\ell$ on the columns $\ell+1,\dots,n-2-\ell$ for $0\le\ell\le(n-9)/2$.
\end{proof}

\section{Derivation of the variance formulas}
\label{sec:app_variance}

\paragraph{Joint moments.}
Let $b$ be a vector with $\sum_vb_v=0$, $p_r=\sum_vb_v^r$, and $z$ a uniformly random permutation of $b$.
For distinct cells, $\E[z_a^4]=p_4/N$, and with $(N)_k=N(N-1)\cdots(N-k+1)$,
\[
\E[z_a^3z_b]=\frac{-p_4}{(N)_2},\quad \E[z_a^2z_b^2]=\frac{p_2^2-p_4}{(N)_2},\quad \E[z_a^2z_bz_c]=\frac{2p_4-p_2^2}{(N)_3},\quad \E[z_az_bz_cz_d]=\frac{3p_2^2-6p_4}{(N)_4},
\]
by M\"obius inversion of the power sums over set partitions.
Because $\sum_v b_v = p_1 = 0$, the odd power sum $p_3$ drops out of the four-index expectation $\E[z_a z_b z_c z_d]$: any partition of $\{a,b,c,d\}$ containing a block of size $3$ must have its remaining block of size $1$, contributing a factor of $p_3 p_1 = 0$.
Write $m_\nu$ for $\E[z_az_bz_cz_d]$ when the equality pattern of $(a,b,c,d)$ is exactly the partition $\nu$ of the positions $\{1,2,3,4\}$, and $m_\lambda=m_\nu$ for any $\nu$ of block type $\lambda$ (e.g.\ $m_{211}$ for $12|3|4$); the displayed moments are $m_4$, $m_{31}$, $m_{22}$, $m_{211}$ and $m_{1111}$.

\paragraph{Evaluation of the M\"obius coefficients in \Cref{prop:var_identity}.}
From the M\"obius inversion on $\Pi_4$ in the proof of \Cref{prop:var_identity}, the coefficients of the four surviving singleton-free partitions $12|34$, $13|24$, $14|23$ and $1234$ are
\[
\begin{aligned}\theta(12|34)=\theta(13|24)=\theta(14|23)&=m_{22}-2m_{211}+m_{1111},\\ \theta(1234)&=m_4-4m_{31}-3m_{22}+12m_{211}-6m_{1111}.\end{aligned}
\]
For the centred values $1,\dots,N$, $p_2=N(N^2-1)/12$ and $p_4=N(N^2-1)(3N^2-7)/240$.
The identities for $p_2, p_4$ give $3p_2^2 - 6p_4 = \frac{1}{240}N(N+1)(5N+7)(N-1)(N-2)(N-3)$, which contains exact factors of $(N-1)(N-2)(N-3)$ cancelling the falling-factorial denominator $(N)_4$, leaving $\E[z_a z_b z_c z_d] = \frac{(N+1)(5N+7)}{240}$; similarly, $2p_4 - p_2^2 = -\frac{1}{720}N(N+1)(5N+7)(N-1)(N-2)(N-3)$ cancels $(N)_3$ in $m_{211}$.
These coefficients simplify to $N(N^2-1)(5N+6)/720$ and $-N^2(N+1)^2/120$; subtracting $\bigl(\frac{N(N+1)}{12}\tr A\bigr)^2$ gives \eqref{eq:var_identity}.

\paragraph{The invariants.}
With $\Gamma_{\ell\ell'}=\langle\tilde\varphi_\ell,\tilde\varphi_{\ell'}\rangle$ and $h_c=A_{cc}=\sum_\ell a_{\ell,c}^2$,
\[
\tr A=\frac1N\tr\Gamma,\qquad\tr A^2=\frac1{N^2}\sum_{\ell,\ell'}\Gamma_{\ell\ell'}^2,\qquad\sum_aA_{aa}^2=\sum_ch_c^2 .
\]
For a line, $(\mathbf 1_L(c)-\frac1n)^2=\frac1{n^2}+(1-\frac2n)\mathbf 1_L(c)$.
Hence $N^2h_c=x_c^2+y_c^2+\frac2{n^2}+(1-\frac2n)(d_{\mathrm m}+d_{\mathrm a})(c)$ for $\Efour$ and $N^2h_c=\frac{2n+2}{n^2}+(1-\frac2n)(2+d_{\mathrm m}(c)+d_{\mathrm a}(c))$ for $\Efull$, where $d_{\mathrm m},d_{\mathrm a}$ are the diagonal indicators.
The entries of $\Gamma$ are given in the proof of \Cref{prop:random_energy_moments}.
The sums needed are $\sum_cx_c^2=\frac{n^2(n^2-1)}{12}$, $\sum_cx_c^4=\frac{n^2(n^2-1)(3n^2-7)}{240}$, $\sum_cx_c^2y_c^2=\bigl(\frac{n(n^2-1)}{12}\bigr)^2$, $\sum_{c\in D}x_c^2=\sum_{c\in D}y_c^2=\frac{n(n^2-1)}{12}$ for either diagonal $D$, $\sum_c(d_{\mathrm m}+d_{\mathrm a})=2n$ and $\sum_c(d_{\mathrm m}+d_{\mathrm a})^2=2n+2\,|D_{\mathrm m}\cap D_{\mathrm a}|$, where $|D_{\mathrm m}\cap D_{\mathrm a}|$ is $1$ for odd and $0$ for even $n$.
Collecting terms gives the following invariants; where two values are given, the first is for odd and the second for even $n$.
For $\Efour$:
\begin{equation}
\begin{aligned}
\tr A&=\frac{(n-1)(n^3+n^2+12)}{6n^4},\\
\tr A^2&=\frac{(n-1)^2(n^6+2n^5+n^4+144)}{72n^8},\qquad\frac{n^8-2n^6+n^4+144n^2-288n+288}{72n^8},\\
\sum_aA_{aa}^2&=\frac{(n-1)(7n^7+7n^6-13n^5+107n^4+240n^2-720n+720)}{180n^{10}},\\
&\qquad\frac{7n^8-20n^6+120n^5-107n^4+240n^3-1320n^2+2880n-2160}{180n^{10}}.
\end{aligned}
\label{eq:inv_low}
\end{equation}
For $\Efull$:
\begin{equation}
\begin{aligned}
\tr A&=\frac{2(n^2-1)}{n^4},\qquad
\tr A^2=\frac{2(n-1)(n^2+n-1)}{n^8},\ \ \frac{2(n^3-2n+2)}{n^8},\\
\sum_aA_{aa}^2&=\frac{2(n-1)(2n-1)(n^2+2n-2)}{n^{10}},\ \ \frac{2(2n^4+n^3-10n^2+12n-6)}{n^{10}}.
\end{aligned}
\label{eq:inv_full}
\end{equation}

\section{Witnesses for \texorpdfstring{$k^*(n)=n-1$, $5\le n\le13$}{k*(n)=n-1, 5<=n<=13}}
\label{sec:app_witnesses}

The arrangement
\[
\begin{pmatrix}1&22&11&24&7\\23&3&12&2&25\\18&6&19&17&5\\4&21&14&10&16\\20&9&15&8&13\end{pmatrix}
\]
has all row sums $M(5)=65$ and column sums $66,61,71,61,66$, i.e.\ $\gamma=(1,-4,6,-4,1)$ and $P_\gamma=(1-t)^4$; it satisfies the degree-$3$ conditions without being semi-magic.
The arrangements below have, in the same way, all row sums $M(n)$ and $\gamma_j=(-1)^j\binom{n-1}j$, so $P_\gamma=(1-t)^{n-1}$, and prove $k^*(n)\ge n-1$ for $6\le n\le13$.
They are also in \texttt{anc/data\_witnesses\_variances.json}; the witnesses for $14\le n\le40$ are in \texttt{anc/kstar\_witnesses.json}.

\medskip
{\scriptsize\setlength{\arraycolsep}{1.25pt}\renewcommand{\arraystretch}{0.95}
\noindent\hfill $n=6$\quad$\left(\begin{array}{@{}rrrrrr@{}}30&26&1&17&2&35\\15&3&23&33&28&9\\20&10&34&24&11&12\\14&22&21&4&19&31\\25&32&6&5&27&16\\8&13&36&18&29&7\end{array}\right)$\hfill $n=7$\quad$\left(\begin{array}{@{}rrrrrrr@{}}43&8&28&17&29&14&36\\30&7&12&33&39&35&19\\13&4&48&16&42&32&20\\22&44&10&9&23&49&18\\1&45&31&38&46&3&11\\41&27&37&2&6&15&47\\26&34&24&40&5&21&25\end{array}\right)$\hfill\null\par\medskip
\noindent\hfill $n=8$\quad$\left(\begin{array}{@{}rrrrrrrr@{}}46&7&49&19&34&2&55&48\\8&64&62&4&3&1&61&57\\54&14&47&26&29&51&28&11\\38&27&42&24&43&37&44&5\\16&53&9&40&35&39&18&50\\52&15&6&59&63&32&21&12\\22&60&10&20&58&36&23&31\\25&13&56&33&30&41&17&45\end{array}\right)$\hfill $n=9$\quad$\left(\begin{array}{@{}rrrrrrrrr@{}}39&14&20&55&65&27&32&41&76\\33&72&71&4&26&50&2&58&53\\5&28&46&69&68&43&62&1&47\\61&13&74&52&10&9&73&56&21\\38&77&51&11&70&36&59&12&15\\19&8&24&3&79&18&80&57&81\\78&40&22&54&44&25&6&63&37\\48&64&60&34&42&30&67&7&17\\49&45&29&31&35&75&16&66&23\end{array}\right)$\hfill\null\par\medskip
\noindent\hfill $n=10$\quad$\left(\begin{array}{@{}rrrrrrrrrr@{}}88&18&48&77&39&93&9&85&14&34\\25&44&82&52&80&24&89&2&70&37\\38&46&63&75&73&36&27&31&54&62\\26&81&90&45&10&53&64&5&57&74\\97&83&7&22&71&6&72&84&12&51\\19&35&66&20&65&47&40&28&91&94\\50&61&76&3&41&79&92&33&55&15\\16&4&87&8&99&17&67&69&60&78\\98&56&21&96&58&13&43&32&59&29\\49&68&1&23&95&11&86&100&42&30\end{array}\right)$\hfill $n=11$\quad$\left(\begin{array}{@{}rrrrrrrrrrr@{}}108&47&96&27&101&8&31&32&98&67&56\\97&50&73&57&106&52&116&29&1&4&86\\40&109&78&12&100&70&117&68&15&16&46\\115&17&37&105&110&21&45&91&76&36&18\\59&71&103&88&6&42&113&22&51&34&82\\120&20&60&14&118&72&92&13&55&49&58\\10&64&35&114&44&89&75&33&93&66&48\\26&112&63&24&107&2&81&5&94&62&95\\9&54&80&25&77&7&19&69&111&121&99\\65&79&30&11&28&53&102&104&39&119&41\\23&38&61&74&84&3&90&85&83&87&43\end{array}\right)$\hfill\null\par\medskip
\noindent\hfill $n=12$\quad$\left(\begin{array}{@{}rrrrrrrrrrrr@{}}65&101&75&47&115&52&138&42&133&33&16&53\\79&98&24&13&122&1&144&23&84&18&121&143\\55&8&64&43&116&22&129&92&139&45&60&97\\9&15&111&62&68&76&136&61&128&102&14&88\\10&86&77&72&93&41&58&31&19&118&131&134\\100&67&57&3&37&11&135&40&117&105&113&85\\90&17&112&63&137&12&106&30&130&108&36&29\\142&141&103&74&104&44&78&28&34&21&99&2\\80&20&96&94&126&50&123&87&4&119&32&39\\140&69&120&81&51&46&71&6&56&7&132&91\\66&110&48&83&124&26&89&95&82&25&73&49\\35&127&38&70&107&27&125&5&109&114&54&59\end{array}\right)$\hfill\null\par\medskip
\noindent\hfill $n=13$\quad$\left(\begin{array}{@{}rrrrrrrrrrrrr@{}}67&120&59&47&117&17&123&8&94&111&167&129&46\\57&4&64&105&124&56&156&41&151&54&96&69&128\\148&10&138&24&130&6&163&12&150&55&127&40&102\\122&65&109&21&147&19&159&26&136&42&80&86&93\\43&52&58&95&165&49&154&33&137&23&146&63&87\\107&99&110&81&60&31&168&1&149&61&39&114&85\\66&143&142&88&38&53&164&11&133&76&62&78&51\\157&118&74&104&119&18&162&7&84&92&112&22&36\\20&161&75&73&116&15&169&34&106&27&100&70&139\\50&77&90&48&152&3&166&16&140&98&68&126&71\\89&82&144&79&145&9&158&13&135&30&101&5&115\\72&125&83&29&153&35&132&14&141&113&45&131&32\\108&37&25&91&134&2&155&97&44&103&28&160&121\end{array}\right)$\hfill\null\par\medskip
}

\end{document}